\documentclass[11pt]{article}
\usepackage[utf8]{inputenc}
\usepackage[margin=1.0in]{geometry}
\usepackage[usenames, dvipsnames]{color}

\usepackage{indentfirst}
\usepackage[tbtags]{amsmath,empheq}
\usepackage{amsmath,bm,bbm}
\usepackage{amsmath}
\usepackage[subnum]{cases}
\usepackage{amssymb}
\usepackage{amsmath}
\usepackage{amsthm}
\usepackage{nomencl}
\usepackage{esvect}

\usepackage{graphicx}
\usepackage{cite,multirow,gensymb}
\usepackage{subfigure}
\usepackage{mwe}
\usepackage{etoolbox}
\usepackage{booktabs}
\usepackage{algorithm}
\usepackage{enumitem}
\usepackage{caption}
\setlist{nolistsep,leftmargin=*}

\usepackage{pdfpages}
\usepackage{rotating}
\usepackage{multicol}
\usepackage{subfigure,textcomp,gensymb,url}
\usepackage{rotating}
\usepackage{psfrag}
\usepackage{multicol}
\usepackage{epstopdf}
\usepackage{graphicx,cite}
\usepackage{algpseudocode}%
\usepackage{anysize, booktabs, setspace}
\usepackage{amsfonts, amssymb, amsthm, eucal}
\usepackage{authblk}
\let\savedegree\degree
\let\degree\relax
\usepackage{mathabx}
\let\degree\savedegree
\usepackage{graphicx}

\usepackage[title]{appendix}

\allowdisplaybreaks[4]

\newtheorem{thm}{Theorem}[section]

\newtheorem{lem}[thm]{Lemma}

\newtheorem{prop}{Proposition}

\newlength{\nomitemorigsep}
\makenomenclature
\renewcommand{\nomname}{}
\renewcommand\nomgroup[1]{%
    \item[\bfseries
    \ifstrequal{#1}{A}{{Sets}}{%
    \ifstrequal{#1}{B}{{Parameters}}{%
    \ifstrequal{#1}{C}{{First-Stage Variables}}{%
    \ifstrequal{#1}{D}{{Second-Stage Variables}}{\ifstrequal{#1}{E}{{Random Variables}}}}}}%
]}

\newcommand{\bsy}[1]{\boldsymbol{#1}}
\newcommand{\bb}[1]{\mathbb{#1}}
\newcommand{\fr}[1]{\mathfrak{#1}}
\newcommand{\wh}[1]{\widehat{#1}}
\newcommand{\mcal}[1]{\mathcal{#1}}
\newcommand{\wt}[1]{\widetilde{#1}}

\title{Distributionally Robust Optimization for a Resilient Transmission Grid During Geomagnetic Disturbances}

\author{
  Mowen Lu$^{\dag}$, Sandra D. Eksioglu$^\dag$, Scott J. Mason$^\dag$, Russell Bent$^{\ddag}$, Harsha Nagarajan$^{\ddag}$ \\
$^\dag$ Department of Industrial Engineering, Clemson University, SC, United States.\\
$^\ddag$ Applied Mathematics and Plasma Physics Group (T-5), \\ Los Alamos National Laboratory, NM, United States.
}
\date{}

\begin{document}

\maketitle
\vspace{-1cm}

\begin{abstract}
In recent years, there have been increasing concerns about the impacts of geomagnetic disturbances (GMDs) on electrical power systems. Geomagnetically-induced currents (GICs) can saturate transformers, induce hot-spot heating and increase reactive power losses. Unpredictable GMDs caused by solar storms can significantly increase the risk of transformer failure. In this paper, we develop a two-stage, distributionally robust (DR) optimization formulation that models uncertain GMDs and mitigates the effects of GICs on power systems through existing system controls (e.g., line switching, generator re-dispatch, and load shedding). This model assumes an ambiguity set of probability distributions for induced geo-electric fields which capture uncertain magnitudes and orientations of a GMD event. We employ state-of-the-art linear relaxation methods and reformulate the problem as a two-stage DR model. We use this formulation to develop a decomposition framework for solving the problem. We demonstrate the approach on the modified Epri21 system and show that the DR optimization method effectively handles prediction errors of GMD events.
\end{abstract}

\providecommand{\keywords}[1]{\textbf{\textit{Key words:}} #1}

\makenomenclature
\nomenclature[A,01]{$\mathcal{N}^{a}, \mathcal{N}^{d}$}{set of nodes in the AC and DC circuit, respectively}
\nomenclature[A,02]{$\mathcal{G}$}{set of generators}
\nomenclature[A,03]{$\mathcal{G}_i$}{set of generators at node $i$}
\nomenclature[A,04]{$\mathcal{I} \subseteq \mathcal{N}^d$}{set of substation neutrals}
\nomenclature[A,05]{$\mathcal{E}^a, \mathcal{E}^d, \mathcal{E}$}{set of edges in the AC and DC circuit, respectively, where $\mathcal{E} = \mathcal{E}^a \cup \mathcal{E}^d$}
\nomenclature[A,06]{$\mathcal{E}^\tau \subseteq \mathcal{E}^a$}{set of transformer edges}
\nomenclature[A,14]{$\mathcal{E}_i$}{set of all edges connected to AC/DC node $i$, where $\mathcal{E}_i = \mathcal{E}_i^+ \cup \mathcal{E}_i^-$}
\nomenclature[A,15]{$\mathcal{E}_i^\tau \subseteq \mathcal{E}^\tau$}{set of DC edges used to compute $d_i^{qloss}$ (as described later) for node $i$}
\nomenclature[B,01]{$c_{i}^0$}{fixed cost for generator $i \in \mathcal{G}$}
\nomenclature[B,02]{$c_{i}^1,c_{i}^2$}{fuel cost coefficients of electricity generation $\rho_i$ from generator $i \in \mathcal{G}$}
\nomenclature[B,03]{$c^{R1}_i, c^{R2}_i$}{fuel cost coefficients of additional electricity generation $\Delta_i$ from  generator $i \in \mathcal{G}$  }
\nomenclature[B,05]{$\kappa_i$}{penalty of under and over-generation at bus $i \in \mcal{N}^a$}
\nomenclature[B,06]{$a_{m}$}{admittance of the grounding line at bus ${m} \in \mathcal{I}$, 0 if bus ${m} \not \in \mathcal{I}$} 
\nomenclature[B,07]{$\alpha_{ij}$}{transformer tap ratio of edge ${e} \in {\mathcal{E}^a}$, $1$ if no transformer is sitting on line $e$}
\nomenclature[B,08]{$a_{e}$}{DC admittance of edge ${e}\in {\mathcal{E}^d} $}
\nomenclature[B,09]{$ r_{e}$, $ x_{e} $}{resistance and reactance of line $e \in {\mathcal{E}^a} $} 
\nomenclature[B,10]{$ g_{e}$, $ b_{e} $}{conductance and susceptance of line $e \in {\mathcal{E}^a}  $}
\nomenclature[B,11]{$ g_{i}$, $ b_{i} $}{shunt conductance and susceptance at bus $ i \in {\mathcal{N}^a} $}
\nomenclature[B,12]{$ d_i^p $, $ d_i^q$}{real and reactive power demand at bus $i \in {\mathcal{N}^a}$}
\nomenclature[B,13]{$ b_{e}^c $}{line charging susceptance of line $e \in {\mathcal{E}^a} $}
\nomenclature[B,14]{$ s_{e} $}{apparent power limit on line $e \in {\mathcal{E}^a} $}
\nomenclature[B,15]{$\overline{\theta}$}{phase angle difference limit}
\nomenclature[B,16]{$ k_{e} $}{loss factor of transformer line ${e} \in \mathcal{E}^{\tau}$}  
\nomenclature[B,17]{$ \underline{v}_i$, $\overline{v}_i$ }{AC voltage limits at bus $i \in {\mathcal{N}^a}$ }
\nomenclature[B,18]{$ \underline{gp}_{i} $, $\overline{gp}_{i} $}{real power generation limits at generator $i \in \mcal{G}$}
\nomenclature[B,19]{$ \underline{gq}_{i} $, $\overline{gq}_{i} $}{reactive power generation limits at generator $i \in \mcal{G}$}
\nomenclature[B,20]{$-\overline{u}^{R}_i, \overline{u}^{R}_i$}{ramp-down and ramp-up limits for generator $i$, respectively}
\nomenclature[B,21]{$\vec{E}$}{the geo-electric field at the area of a transmission system}
\nomenclature[B,22]{$L^N_e$,$L^E_e$}{the northward and eastward components of the displacement of each transmission line $e \in \mcal{E}^d$, respectively}
\nomenclature[C,01]{$ z^g_{i} $}{1 if generator { $i \in \mathcal{G}$} is online; 0 otherwise}
\nomenclature[C,02]{$ z^a_{e} $}{1 if line { $e \in \mathcal{E}^a$} is switched on; 0 otherwise}
\nomenclature[C,03]{$\rho_i$}{reserved real-power output of generator $i \in \mathcal{N}^a$}
\nomenclature[D,01]{$\theta_i$}{phase angle at bus $i \in{\mathcal{N}^a}$}
\nomenclature[D,02]{$v_i$}{voltage magnitude at bus $i \in {\mathcal{N}^a }$}
\nomenclature[D,03]{$v_i^d$}{induced DC voltage magnitude at bus $i \in {\mathcal{N}^d}$}
\nomenclature[D,04]{$l_{e}^d$}{GIC flow on transformer line ${e} \in \mathcal{E}^{\tau} $}
\nomenclature[D,05]{$I_{e}^d$}{the effective GIC on transformer line ${e}\in \mathcal{E}^{\tau} $}
\nomenclature[D,06]{$ d_{i}^{qloss}$}{GIC-induced reactive power loss at bus $i\in \mathcal{N}^a$}
\nomenclature[D,07]{$ \Delta_{i}^{p}$}{{excess} real power generated at generator $i \in \mcal{G}$}
\nomenclature[D,08]{$ p_{ij}$, $q_{ij}$}{real and reactive power flow on line $ e_{ij} \in {\mathcal{E}^a}$, as measured at node $i$}
\nomenclature[D,09]{$f_i^p$, $f_i^q$}{real and reactive power generated at bus $i \in \mathcal{N}^a$}
\nomenclature[D,10]{$l_i^{p+}$, $l_i^{q+}$}{real and reactive power shed at bus $i \in \mathcal{N}^a$}
\nomenclature[D,11]{$l_i^{p-}$, $l_i^{q-}$}{real and reactive power over-generated at bus $i \in \mathcal{N}^a$}
\nomenclature[F,01]{$\widetilde{\nu}^E$, $\widetilde{\nu}^N$}{GMD-induced geo-electric fields in Eastward and Northward, respectively}
\nomenclature[F,02]{$\widetilde{\nu}_e^d$}{induced voltage sources on line $e \in \mcal{E}^d$, as a function of $\widetilde{\nu}^E$ and $\widetilde{\nu}^N$}
\renewcommand\nomname{Nomenclature}
\printnomenclature[0.9in]

\section{Introduction}
Solar flares and coronal mass ejections form solar storms where charged particles escape from the sun, arrive at Earth, and cause geomagnetic disturbances (GMDs). GMDs lead to changes in the Earth’s magnetic field, which then create geo-electric fields. These low-frequency, geo-electric fields induce quasi-DC currents, also known as Geomagnetically-Induced Currents (GICs), in grounded sections of electric power systems \cite{albertson1973solar,albertso1974effects, albertson1993geomagnetic}. GICs are superimposed on the existing alternating currents (AC) and bias the AC in transformers. This bias leads to half-cycle saturation and magnetic flux loss in regions outside of the transformer core. The energy stored in the stray flux increases the reactive power consumption of transformers, which can lead to load shedding since, in general, generators are not designed to handle such unexpected losses. In addition, the stray flux also drives eddy currents that can cause excessive transformer heating. Excess heating leads to reduced transformer life and, potentially, immediate damage \cite{GICeffects2012}. As a result, when GMD events occur on large-scale electric power systems, the resulting power outages can be catastrophic. For example, the GMD event in Quebec in 1989 led to the shutdown of the Hydro-Quebec power system. As a consequence, six million people 
lost access to power for nine hours. By some estimates the net cost of this event was \$13.2 million, with damaged equipment accounting for \$6.5 million of the cost \cite{bolduc2002gic}.

The potential GMD impacts to transformers in the bulk electric power system have motivated the United States government to sponsor research to improve understanding of GMD events and identify strategies for mitigating the impacts of GMDs on power systems\cite{exec2016,FEDERALENERGYREGULATORYCOMMISSION2016}. To model the potential risks introduced by GICs, research institutes and  electric power industries have actively improved GIC modeling and GIC monitoring \cite{erinmez2002management,cannon2013extreme,qiu2015geomagnetic,overbye2013power, horton2012test}. For example, the North American Electric Reliability Corporation (NERC) developed a procedure to quantify GICs in a system based on the exposure to geo-electric field \cite{GIC2013flow}. These models are used to conduct risk analyses which estimate the sensitivity of reactive power losses from GICs. These studies indicate that risk mitigation warrants further study.

The recent literature mainly focuses on mitigating two risks introduced by GICs. The first is voltage sag caused by increased reactive power consumption in transformers \cite{GICeffects2012}. The second is transformer damage caused by excessive hot-spot thermal heating\cite{GICcapacity}. One mitigation is DC-current blocking devices that prevent the GIC from accessing transformer neutrals \cite{bolduc2005development}. Unfortunately, these devices are expensive; a single unit can cost \$500K \cite{liang2015optimal,zhu2015blocking,kovan2015mitigation}. The high cost is a barrier to wide-scale adoption of this technology. To address this issue, Lu \textit{et al.} \cite{lu2017optimal} developed a GIC-aware optimal AC power flow (ACOPF) model that uses existing topology control (line switching) and generator dispatch to mitigate the risks of GIC. This work showed that topology reconfiguration can effectively protect power systems from GIC impacts, which was also later observed in \cite{kazerooni2018transformer}. However, these papers assumed a deterministic GMD event (i.e., the induced geo-electric field is known). In reality, solar storms, such as solar flares and coronal mass ejections (CMEs), are difficult to predict. Although ground- and space-based sensors and imaging systems have been utilized by the National Aeronautics and Space Administration (NASA) to observe these activities at various depths in the solar atmosphere, the intensity of a storm cannot be measured until the released particles reach Earth and interact with its geomagnetic field \cite{NASA}. As a result, there is often uncertainty in predictions of solar storms and the induced geo-electric field, which introduce operational challenges for mitigating the potential risks by GIC to power systems. 



This paper extends the study of GIC mitigation of \cite{lu2017optimal} in two ways. The model (1) uses line switching and generator ramping to mitigate GIC impacts and (2) assumes uncertain magnitude and orientation of a GMD. We formulate this problem as a two-stage distributionally robust (DR) optimization model that we refer to as DR-OTSGMD. The first-stage problem  selects a set of transmission lines and generators to serve power in a power system ahead of an imminent GMD event. The second-stage problem evaluates the performance of the network given the status of transmission lines, the reserved electricity for daily power service, and a realized GMD. In this model, the uncertain magnitude and direction of a GMD are modeled with an ambiguity set of probability distributions of the induced geo-electric field. The objective is to minimize the expected total cost of the worst-case distribution defined in the ambiguity set. Instead of assuming a specific candidate distribution, the ambiguity set is described by statistical properties of uncertainty, such as the support and moment information. As a result, the DR optimization approach yields less conservative solutions than prevalent robust optimization approaches \cite{xiong2017distributionally, zhao2012robust}. In comparison to stochastic programming, the DR method does not require complete information of the exact probability distribution, i.e., the decisions do not rely on assumptions about unknown distributions \cite{bansal2017decomposition}. Additionally, the reliability of the solutions found with DR methods does not require a large number of random samples, which can lead to computational tractability and scalability issues for large-scale problems \cite{li2016distributionally}. 
In the literature, DR optimization approaches have been used to model a variety of problems \cite{scarf1957min, delage2010distributionally, fabozzi2010robust, jiang2016data, goh2010distributionally, dupavcova1987minimax}, such as contingency-constrained unit commitment \cite{zhao2018distributionally, chen2018distributionally, xiong2017distributionally}, optimal power flow with uncertain renewable energy generation \cite{zhang2017distributionally, lubin2016robust, xie2018distributionally, li2016distributionally}, planning and scheduling of power systems \cite{bian2015distributionally, wang2016distributionally}, and energy management \cite{qiu2015distributionally, wei2016distributionally}. These observations and the lack of studies about probability distributions for GMD events makes a DR model an attractive choice for modeling uncertainty in this problem.
Given the computational complexity (nonlinear physics and trilevel modeling) of this problem, we focus on developing methods for
computing DR operating points based on lower bounding techniques.
Techniques for computing upper bound solutions and the globally optimal solution are left for future work. The main contributions of this paper include: 

\begin{itemize}
    \item A two-stage DR-OTSGMD model that captures the uncertainty of the GMD-induced geo-electric field. This formulation considers the AC physics of power flow and reactive power consumption at transformers due to uncertain GICs.
    \item {\color{Black} A relaxation and reformulation of the two-stage DR-OTSGMD problem which facilitates the problem decomposition and application of efficient decomposition algorithms. We show that solutions of the resulting problem are a valid lower bound to the original problem. }
    \item Extensive numerical analysis using the Epri21 test system to validate the model and demonstrates its effectiveness in generating robust solutions.
\end{itemize}

The remainder of this paper is organized as follows. In section \ref{sec:formulation}, we discuss the two-stage DR-OTSGMD formulation and the ambiguity set of probability distributions for the induced geo-electric field. We describe linear relaxations of nonlinear and non-convex functions in the problem, then derive a decomposition framework for solving the relaxation. In Section \ref{sec:methodology}, we present a column-and-constraint generation algorithm and demonstrate the effectiveness of the approach using the Epri21 test system in Section \ref{Sec:case study}. Finally, we conclude and provide directions for future research in Section \ref{sec:conclusions}.

\section{Problem Formulation}
\label{ProblemFormulation}

Here, we describe the DR-OTSGMD. We first introduce the AC power flow model. We then introduce the modeling of the GIC. We note that the AC and GIC flows exist on the same physical power system, however, they require different representations of that system. The two representations are tied through defined interactions of the AC and GIC currents.
Third, we discuss the coupling between the AC power flow model and the GIC. Next, we discuss the uncertainty model. We then discuss our convex (linear) relaxation of the problem and reformulate this relaxation as a distributionally robust optimization problem. 

{
\color{Black}

\subsection{Electric Power Model--AC}

The objective function for operating the AC electric power systems minimizes total generation dispatch cost as formulated by a quadratic function \eqref{nl_obj}. In this cost function, $f^p_i$ denotes the real power generation from generator $i$. The fixed generation cost is incurred when a generator is switched on (i.e., $z^g_{i}=1$).
\begin{subequations}
\begin{align}
    &\label{nl_obj} \min \sum_{i \in \mcal{G}} z^g_{i}c^0_i + c^1_if^p_i + c^2_i(f^p_i)^2 
\end{align}
\end{subequations}

The  AC  physics  of  electric  power  systems  are  governed by  Kirchoff’s  and  Ohm’s  laws. Here, Kirchoff's law is modeled as constraints \eqref{nl_pbalance}-\eqref{nl_qbalance} and Ohms law is modeled as constraints \eqref{nl_pij}-\eqref{nl_qji}.
\begin{subequations}
\allowdisplaybreaks
\label{nonlinear_ac_kirchoff}
\begin{align}
& \label{nl_pbalance} \smashoperator{\sum_{e_{ij} \in \mcal{E}_i} } p_{ij} = \sum_{k\in \mathcal{G}_i} f^p_{k}-d^p_{i} - v_i^2g_i \qquad \forall i\in \mcal{N}^a \\
&\label{nl_qbalance} \smashoperator{\sum_{e_{ij}\in \mcal{E}_i}} q_{ij}
= \sum_{k\in\mathcal{G}_i}f^q_{k}-d^q_{i}  + v_i^2b_i \qquad \forall i \in \mcal{N}^a \\
& \label{nl_pij} p_{ij}=\frac{1}{\alpha_{ij}^2}g_{e}v_{i}^2-\frac{1}{\alpha_{ij}}v_{i}v_{j}\big(g_{e}\cos(\theta_{i}-\theta_{j}) +b_{e}\sin(\theta_{i}-\theta_{j})\big) \qquad \forall e_{ij} \in \mcal{E}^a \\
& \label{nl_qij}  q_{ij}=-\frac{1}{\alpha_{ij}^2}(b_{e}+\frac{b_{e}^c}{2})v_{i}^2 + \frac{1}{\alpha_{ij}}v_{i}v_{j}\big( b_{e}\cos(\theta_{i}-\theta_{j}) - g_{e}\sin(\theta_{i}-\theta_{j})\big) \qquad \forall e_{ij} \in \mcal{E}^a \\
& \label{nl_pji}  p_{ji}=g_{e}v_{j}^2-\frac{1}{\alpha_{ij}}v_{i}v_{j}\big(g_{e}\cos(\theta_{j}-\theta_{i}) +b_{e}\sin(\theta_{j}-\theta_{i})\big) \qquad \forall e_{ij} \in \mcal{E}^a \\
& \label{nl_qji}  q_{ji}=-(b_{e}+\frac{b_{e}^c}{2})v_{j}^2 + \frac{1}{\alpha_{ij}}v_{i}v_{j}\big(b_{e}\cos(\theta_{j}-\theta_{i})-g_{e}\sin(\theta_{j}-\theta_{i})\big) \qquad \forall e_{ij} \in \mcal{E}^a
\end{align}
\end{subequations}%

\noindent
Here, $p_{ij}$ and $q_{ij}$ are the real and reactive flow between buses $i$ and $j$, as measured at node $i$, respectively. Similarly, $f_j^p$ and $f_j^q$ are the real and reactive generation at generator $j$ and $d_i^p$ and $d_i^q$ are the real and reactive load at $i$. Finally, $v_i$ and $\theta_i$ are the voltage magnitude and phase angle at bus $i$, respectively.

The flow on lines is restricted by the physical limits of the grid and are modeled with constraints \eqref{vi}-\eqref{nl_capacity}. Constraints (\ref{vi}) limit the voltage magnitude at buses, while constraints  (\ref{thetaij_ub}) apply bounds on the phase angle difference between two buses. Constraints (\ref{nl_capacity}) model operational thermal limits of lines at both sides.
\begin{subequations}
\begin{align}
& \label{vi} \underline{v}_{i}\leq v_{i} \leq \overline{v}_{i} \qquad \forall i \in \mcal{N}^a \\
& \label{thetaij_ub} |\theta_{i}-\theta_{j}| \leq \overline{\theta} \qquad \forall e_{ij} \in \mcal{E}^a \\
& \label{nl_capacity} p_{ij}^2+q_{ij}^2 \leq s^2_{e}, \;\;\;p_{ji}^2+q_{ji}^2 \leq s^2_{e} \qquad \forall e_{ij}  \in \mcal{E}^a 
\end{align}
\end{subequations}

Generator outputs are limited by the commitment of that generator and capacity upper and lower bounds which are modeled as constraints \eqref{gp_za}-\eqref{zg}. 
Discrete variables $z_i^g$ indicate the generator’s on-off status and are incorporated into generator power limits.
\begin{subequations}
\begin{align}
& \label{gp_za}  z^g_{i}\underline{gp}_i \leq f_{i}^p \leq z^g_{i}\overline{gp}_i \qquad \forall i \in \mathcal{G} \\
& \label{gq_za}  z^g_{i}\underline{gq}_i \leq f_{i}^q \leq z^g_{i}\overline{gq}_i \qquad \forall i \in \mathcal{G} \\
& \label{zg} z^g_{i} \in \{0, 1\} \qquad \forall i \in \mcal{G}
\end{align}
\end{subequations}

For mitigating GIC effects, transmission lines are allowed to be switched on or off.  We model switching by modifying constraints \eqref{nl_pji}-\eqref{nl_qji} and \eqref{thetaij_ub}-\eqref{nl_capacity} with:
\begin{subequations}
\begin{align}
& \label{z_nl_pij} p_{ij}=z^a_{e}\Big(\frac{1}{\alpha_{ij}^2}g_{e}v_{i}^2-\frac{1}{\alpha_{ij}}v_{i}v_{j}\big(g_{e}\cos(\theta_{i}-\theta_{j}) +b_{e}\sin(\theta_{i}-\theta_{j})\big)\Big) \qquad \forall e_{ij} \in \mcal{E}^a \\
& \label{z_nl_qij}  q_{ij}=z^a_{e}\Big(-\frac{1}{\alpha_{ij}^2}(b_{e}+\frac{b_{e}^c}{2})v_{i}^2 + \frac{1}{\alpha_{ij}}v_{i}v_{j}\big( b_{e}\cos(\theta_{i}-\theta_{j}) - g_{e}\sin(\theta_{i}-\theta_{j})\big)\Big) \qquad \forall e_{ij} \in \mcal{E}^a  \\
& \label{z_nl_pji}  p_{ji}=z^a_{e}\Big(g_{e}v_{j}^2-\frac{1}{\alpha_{ij}}v_{i}v_{j}\big(g_{e}\cos(\theta_{j}-\theta_{i}) + b_{e}\sin(\theta_{j}-\theta_{i})\big)\Big) \qquad \forall e_{ij} \in \mcal{E}^a  \\
& \label{z_nl_qji}  q_{ji}=z^a_{e}\Big(-(b_{e}+\frac{b_{e}^c}{2})v_{j}^2 + \frac{1}{\alpha_{ij}}v_{i}v_{j}\big(b_{e}\cos(\theta_{j}-\theta_{i})-g_{e}\sin(\theta_{j}-\theta_{i})\big)\Big) \qquad \forall e_{ij} \in \mcal{E}^a \\
& \label{z_thetaij_ub} z^a_e|\theta_{i}-\theta_{j}| \leq \overline{\theta} \qquad \forall e_{ij} \in \mcal{E}^a \\
& \label{z_nl_capacity} p_{ij}^2+q_{ij}^2 \leq z^a_{e}s^2_{e}, \;\;\;p_{ji}^2+q_{ji}^2 \leq z^a_es^2_{e} \qquad \forall e_{ij}  \in \mcal{E}^a \\
& \label{z_a} z^a_e \in \{0,1\}\qquad \forall e \in \mcal{E}^a
\end{align}
\end{subequations}

\noindent where $z_e^a$ denotes the on-off status of transmission line $e$ in the AC model. Constraints \eqref{z_nl_pij}-\eqref{z_nl_qji} model AC power flow on each line when the line is closed and force the flow to zero when the line is open. Similarly, constraints \eqref{z_thetaij_ub} and \eqref{z_nl_capacity} place phase angle difference limits and thermal limits on active lines, respectively.

Given that generators can respond in a limited way to GIC, we model this response with ramping constraints \eqref{setpoint}-\eqref{ramp}. Here, $\rho_i$ is the setpoint of generator $i \in \mcal{G}$. Constraints \eqref{ramp} limit the deviation of real power generation from $\rho_i$ when generator $i$ is online.
\begin{subequations}
\begin{align}
&\label{setpoint} z^g_i\underline{gp}_i \leq \rho_i \leq z^g_i\overline{gp}_i \qquad i \in \mcal{G} \\
& \label{ramp}  |f_i^p - \rho_i| \leq z^g_i\overline{u}_i^{R}\overline{gp}_i  \qquad \forall i \in \mathcal{G}
\end{align}
\end{subequations}

\noindent
The objective function is then reformulated with the generator set points, i.e.,
\begin{equation}
    \label{nl_obj_setpoint} \min \sum_{i \in \mcal{G}} z^g_{i}c^0_i + c^1_i\rho_i + c^2_i(\rho_i)^2 
\end{equation}
\noindent
In addition, to ensure feasibility at all times, we introduce slack variables for over (i.e., $l^{p+}_i, l^{q+}_i$) and under (i.e., $l^{p-}_i, l^{q-}_i$) load consumption. This changes equations \eqref{nl_pbalance}-\eqref{nl_qbalance} to \eqref{nl_pbalance_slack}-\eqref{nl_qbalance_slack}.
\begin{subequations}
\begin{align}
& \label{nl_pbalance_slack} \smashoperator{\sum_{e_{ij} \in \mcal{E}_i} } p_{ij} = \sum_{k \in \mathcal{G}_i} f^p_{k}-d^p_{i} + l_i^{p+} - l_i^{p-} - v_i^2g_i \qquad \forall i\in \mcal{N}^a \\
&\label{nl_qbalance_slack} \smashoperator{\sum_{e_{ij}\in \mcal{E}_i}} q_{ij}
= \sum_{k \in \mathcal{G}_i} f^q_{k}-d^q_{i} + l_i^{q+} - l_i^{q-} + v_i^2b_i \qquad \forall i \in \mcal{N}^a 
\end{align}
\end{subequations}

\noindent Similarly, the objective function is also modified to penalize the slackness with high cost ($\kappa_i$), i.e.,
\begin{equation}
    \label{nl_obj_slack} \sum_{i \in \mcal{G}} z^g_{i}c^0_i + c^1_i\rho_i +  c^2_i(\rho_i)^2  + \sum_{i \in \mcal{N}^a} \kappa_i (l^{p+}_{i}+l^{q+}_{i} + l^{p-}_{i}+l^{q-}_{i}) 
\end{equation}

\noindent
The slack is interpreted as an indication for the need to shed loads or to over generate.










\subsection{Electric Power Model--GIC}
}

\subsubsection{$\widetilde{\nu}_e^d$ calculation}\label{subsubsection:geo-electric field calc}
The GIC calculation depends on induced voltage sources $(\widetilde{\nu}_e^d)$ on each power line, $e \in \mathcal{E}^d$, in the network which are determined by the geo-electric field integrated along the transmission line. This relationship is modeled in Eq.~(\ref{eq:dc_source})
\begin{equation}\label{eq:dc_source}
\widetilde{\nu}_e^d=\oint \vec{E}_e\cdot d\vec{l}_e,
\end{equation} where $\vec{E}_e$ is the geo-electric field in the area of transmission line $e \in \mcal{E}^d$, and $d\vec{l}_e$ is the incremental line segment length including direction \cite{GIC2013flow}. In practice, the actual geo-electric field varies with geographical locations. In this paper, we use a common assumption that the geo-electric field in the geographical area of a transmission line is uniformly distributed \cite{horton2012test,zhu2015blocking,GIC2013flow}, i.e., $\vec{E} = \vec{E}_e$ for any $e\in \mcal{E}^d$. Hence, only the coordinates of the line end points are relevant and $\vec{E}$ is resolved into its eastward (x axis) and northward (y axis) components \cite{GIC2013flow}. Given $\vec{L}_e$, the length of line $e$ with direction, equation (\ref{eq:dc_source}) is reformulated as:
\begin{equation}\label{eq:voltage source}
 \widetilde{\nu}_e^d = \vec{E}\cdot\vec{L}_e=\widetilde{\nu}^NL_e^N + \widetilde{\nu}^EL_e^E, \quad \forall e \in \mathcal{E}^d
\end{equation}
where $\widetilde{\nu}^N$ and $\widetilde{\nu}^E$ represent the geo-electric fields (V/km) in the northward and eastward directions, respectively. $L_e^N$ and $L_e^E$ denote {distances (km) along the northward and eastward directions}, respectively, which depend on the geographical location (i.e., latitude and longitude) \cite{GIC2013flow}.


\subsubsection{Transformer modeling} 
{\color{Black} 
{\color{black}The two most common transformers in electrical transmission systems are generator step-up (GSU) transformers and network transformers. GSUs connect the output
terminals of generators to the transmission network. In contrast, network transformers are generally located relatively far from generators and transform voltage between different sections of the transmission system.} In this section, we discuss how to model these two classes of transformers in the DC circuit induced by a GMD. For notation brevity, we use subscripts $h$ and $l$ to represents the high-voltage (HV) and low-voltage (LV) buses of all transformers.  Additionally, we define $N_x$ and $I^d_x$ as the number of turns and GICs in the transformer winding $x$, respectively, and let $\Theta(\cdot)$ denote a linear function of $I^d_x$. 

In this paper, we model two types of GSU transformers: (1) GWye-Delta and (2) Delta-GWye. Consistent with common engineering practice, we assume that each GSU is grounded on the HV side that connects the generator to the transmission network. Figure \ref{fig:GSU} shows the DC representation of the GSU transformer. It shows that the effective GICs (denoted as $\widetilde{I}^d$) depend on GICs in the HV primary winding, i.e.,
\begin{equation}
\label{GSU_eff}
   \widetilde{I}^d = \Big|\Theta(I^d_h) \Big| = \Big|I^d_h\Big|
\end{equation}

For network transformers, we model (1) two-winding transformers: GWye-GWye Anto- and Gwye-Gwye transformers and (2) three-winding transformers including Delta-Delta, Delta-Wye, Wye-Delta, and Wye-Wye configurations. The DC-equivalent circuits of GWye-GWye Anto- and Gwye-Gwye transformers are shown in Figures \ref{fig:auto} and \ref{fig:cov}, respectively, where subscripts $s$ and $c$ denote the series and common sides of the auto-transformer.}
For an auto-transformer, as shown in Figure \ref{fig:auto}, the effective GICs are determined by GIC flows in the series and common windings, i.e.,
%
\begin{equation}
\label{auto_eff}
    \widetilde{I}^d = \Big|\Theta(I^d_s) + \Theta(I^d_c)\Big| = \Big|\frac{(\alpha-1)I^d_s + I^d_c}{\alpha}\Big|
\end{equation}
\noindent where turns ratio $\alpha = \frac{N_s + N_c}{N_c}$. Similarly, the effective GICs for a Gwye-Gwye transformer (Figure \ref{fig:cov}) models GICs on the HV and LV sides, i.e.,
\begin{equation}
\label{gwye_eff}
    \widetilde{I}^d = \Big|\Theta(I^d_h) + \Theta(I^d_l)\Big| = \Big|\frac{\alpha I^d_h + I^d_l}{\alpha}\Big| 
\end{equation}
where transformer turns ratio $\alpha=\frac{N_h}{N_l}$. 

For all three-winding transformers we assume that their windings are ungrounded. Hence, their associated effective GICs are zero because ungrounded transformers do not provide a path for GIC flow \cite{GIC2013flow}.
\begin{figure}[!h]
   \centering
   \captionsetup{font=small}
   \subfigure[GWye-Delta/Delta-GWye GSUs]{
   \label{fig:GSU}
   \includegraphics[scale=1.46]{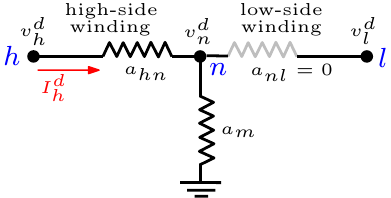}} 
   \subfigure[GWye-GWye Auto]{
   \label{fig:auto}
   \includegraphics[scale=1.46]{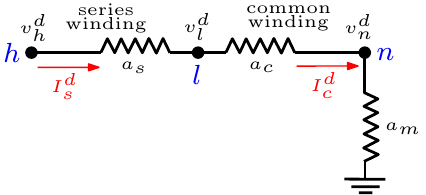}} 
   \subfigure[Gwye-Gwye]{
   \label{fig:cov}
   \includegraphics[scale=1.46]{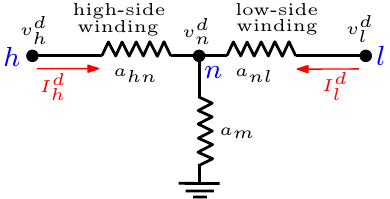}}       
   \caption[DC equivalent circuits for different types of transformers]{DC equivalent circuits for different types of transformers. $h$, $l$, and $n$ (blue font) represent high-side bus, low-side bus, and neutral bus, respectively.} 
   \label{Fig:xf_dc}
\end{figure}
{\color{Black}
\subsubsection{GIC Modeling}

The GMD-induced equivalent DC circuit is modeled through constraints \eqref{dc_balance}-\eqref{dc_ij}. These constraints calculate the GIC flow on each line in the DC represenation of the physical system. Equations \eqref{dc_balance} model Kirchoff's law and equations \eqref{dc_ij} model Ohm law, where $\wt{\nu}^d_e$ denotes the DC-voltage source of line $e$ induced by the geo-electric field (see Section 2.2.1). 
\begin{subequations}
\begin{align}
&\label{dc_balance} \sum_{e \in \mcal{E}_m}I_e^d = a_mv_m^d \qquad \forall m \in \mcal{N}^d \\
& \label{dc_ij} I_e^d = a_e(v^d_m - v^d_n + \wt{\nu}^d_e) \qquad \forall e_{mn} \in \mcal{E}^d
\end{align}
\end{subequations}

Constraints \eqref{Idmag}-\eqref{Idub} model the effective GIC value of each transformer, $\wt{I}^d_e$, which depends on the type of transformer $e \in \mcal{E}^\tau$. Instead of introducing additional discrete variables, constraints \eqref{Idmag} model and relax the \textit{absolute value} of $\smashoperator{\sum_{\wh{e}_{ij} \in \mcal{E}^w_{e}}}\Theta(I_{\wh{e}}^d)$  defined in (\ref{GSU_eff})--(\ref{gwye_eff}). 
This relaxation is often tight as larger DC current typically causes problems and drives the objective value up.
Constraint set (\ref{Idub}) denotes the maximum allowed value of GIC flowing through transformers. We assume this limit is twice the upper bound of the AC in the network.
\begin{subequations}
\label{eq:effective_GIC}
\begin{align}
& \qquad \quad \label{Idmag} \widetilde{I}_{{e}}^d \geq \smashoperator{\sum_{\wh{e}_{ij} \in \mcal{E}^w_{e}}}\Theta(I_{\wh{e}}^d), \quad \widetilde{I}_{{e}}^d \geq -\smashoperator{\sum_{\wh{e}_{ij} \in \mcal{E}^w_{e}}}\Theta(I_{\wh{e}}^d) \quad \forall e \in \mcal{E}^\tau \\
& \label{Idub} \qquad \quad 0 \leq \wt{I}_{e}^d\leq \max_{\forall \wh{e} \in \mcal{E}^a} 2\overline{I}_{\hat{e}}^a \qquad \forall e \in \mcal{E}^{\tau} 
\end{align}
\end{subequations}

Constraint set (\ref{qloss}) computes the reactive power load due to transformer saturation \cite{albertson1973solar,overbye2013power, zheng2014effects, zhu2015blocking} by using the effective GICs for each transformer type. Here, $d^{qloss}_i$ denotes the total reactive power loads at bus $i$ which is the high-voltage side of transformer $e \in \mcal{E}^{\tau}_i$.
\begin{subequations}
\begin{align}
& \label{qloss} \qquad \quad d_{i}^{qloss}=\smashoperator{ \sum_{e \in \mcal{E}_i^{\tau}}}k_{e}v_{i}\wt{I}^d_{e} \qquad \forall i \in \mcal{N}^a
\end{align}
\end{subequations}

Similar to the AC model, lines are switched on and off by modifying constraints \eqref{dc_ij} with on-off variables $z^d_e$ as follows. 
%
\begin{subequations}
\begin{align}
& \label{dc_ij_z} I_e^d = z_e^da_e(v^d_m - v^d_n + \wt{\nu}^d_e) \qquad \forall e_{mn} \in \mcal{E}^d \\
& z^d_e \in \{0,1\}\qquad \forall e \in \mcal{E}^d
\end{align}
\end{subequations}

\subsection{Electric Power Model--Coupled GIC and AC}

The GIC and AC are tied in two ways.  First, switching lines on and off impact both the AC and GIC formulation.  We use the AC switching variables (i.e., $z_e^a$ and $z_i^g$) and tie these variables to GIC switching by linking an edge $e \in \mcal{N}^d$ in the DC circuit to an edge in the AC circuit and dropping the $z_e^d$ variables. 
Specifically, for 
$e \in \mathcal{E}^d$, we use
$\overrightarrow{e}$ to denote the associated AC line of $e$.  This is a one-to-one mapping for transmission lines and a many-to-one mapping for transformers \cite{lu2017optimal}. 
This is a one-to-one mapping for transmission lines and a many-to-one mapping for transformers (i.e., both high and low-side windings of a transformer are linked to one and only one line in the AC circuit). Using this notation, constraints \eqref{dc_ij_z} are modified to constraints \eqref{dc_ij_zag}.
\begin{equation}
\label{dc_ij_zag}
 I_e^d = z^a_{\overrightarrow{e}}a_e(v^d_m - v^d_n + \wt{\nu}^d_e) \qquad \forall e_{mn} \in \mcal{E}^d 
\end{equation}

Second, the reactive losses induced by the GIC are added to the AC Kirchoff equations by modifying constraints \eqref{nl_qbalance_slack} to \eqref{nl_qbalance_slack_qloss}.
\begin{equation}
\label{nl_qbalance_slack_qloss} \smashoperator{\sum_{e_{ij}\in \mcal{E}_i}} q_{ij}
= \sum_{k \in \mathcal{G}_i} f^q_{k}-d^q_{i} - d^{qloss}_i + l_i^{q+} - l_i^{q-} - v_i^2b_i \qquad \forall i \in \mcal{N}^a 
\end{equation}

\subsection{A Two-stage OTSGMD Model}
\label{subsec:DRGMDOTS}


Assuming that the induced geo-electric field is constant (i.e., $\wt{\nu}^d_e$ are known parameters), the OTSGMD problem is deterministic and it is natural to formulate the problem with two stages. In the first stage, generator setpoints and the on/off status of generators and edges are determined prior to an imminent GMD event. The second stage determines operational decisions of the power system in response to the given event. The formulation of this two-stage OTSGMD model is presented with:
\begin{subequations}
\label{first}
\begin{align}
&\label{first_obj} \min \sum_{i \in \mcal{N}^g} z^g_ic^0_i +  c_i^1\rho_i + c_i^2(\rho_i)^2  + \mcal{H}(\bsy{z^a},\bsy{z^g}, \bsy{\rho}, \bsy{\widetilde{\nu}^d})\\
& \quad \textit{s.t.} \quad \eqref{zg}, \eqref{z_a}, \eqref{setpoint}
\end{align}
\end{subequations}
\vspace{-1.3cm}
\begin{subequations}
\label{second}
\begin{align}
& \text{Where} \label{second_obj}\quad \mcal{H}(\bsy{z^a}, \bsy{z^g}, \bsy{\rho}, \bsy{\widetilde{\nu}^d}) = \min \sum_{i \in \mcal{N}^g} c_i^{R2}(\Delta^p_i)^2 +  c_i^{R1}\Delta^p_i + \sum_{i \in \mcal{N}^a} \kappa_i (l^{p+}_{i}+l^{q+}_{i} + l^{p-}_{i}+l^{q-}_{i}) \\
& \qquad \qquad \qquad \qquad \textit{s.t.}\label{ramp_up}\qquad\quad  \Delta_i^p \geq 0, \ \ \Delta_i^p \geq f_i^p - \rho_i \qquad \forall i 
\in \mathcal{G}  \\
& \qquad \qquad \qquad \qquad \qquad \qquad \eqref{vi},
\eqref{gp_za}-\eqref{gq_za},
\eqref{z_nl_pij}-\eqref{z_nl_capacity}, \eqref{ramp},
\eqref{nl_pbalance_slack},
\eqref{eq:voltage source}, \nonumber \\ 
& \qquad \qquad \qquad \qquad \qquad \qquad \eqref{dc_balance}, \eqref{Idmag}-\eqref{qloss}, \eqref{dc_ij_zag}-\eqref{nl_qbalance_slack_qloss}
\end{align}
\end{subequations}

The first-stage problem \eqref{first} specifies a system topology and  real-power generation set points for daily electricity consumption. The objective function \eqref{first_obj} minimizes the total cost including the fixed cost for keeping generators on, the fuel cost for the settled generator outputs, and the operating cost $\mcal{H}(\bsy{z^a},\bsy{z^g}, \bsy{\rho}, \bsy{\widetilde{\nu}^d})$ of the second stage. The second-stage problem is formulated as \eqref{second} which minimizes $\mcal{H}(\bsy{z^a},\bsy{z^g}, \bsy{\rho}, \bsy{\widetilde{\nu}^d})$ as function of the fuel cost for the ramp-up generation and a penalty for over and under load. Constraints \eqref{ramp_up} model the ramp-up generation (i.e., positive deviation) of real power at each generator. If generator $i \in \mcal{G}$ ramps up, the minimum $\Delta^p_i$ is equal to $f_i^p-\rho_i$; 0 otherwise. Note that the cost function \eqref{second_obj} takes into account only the ramp-up power generation, thus there is no cost if generator $i \in \mcal{G}$ ramps down.

\subsection{A Two-stage DR-OTSGMD Model with Uncertainty}


In practice, solar storms are unpredictable, which in turn leads to uncertain geo-electric fields $\vec{E}$. Specifically, the northward and eastward geo-electric fields (V/km), $\wt{\nu}^N$ and $\wt{\nu}^E$, are random variables and affect the induced voltage sources on transmission lines $\nu^d_e$ as equation \eqref{eq:voltage source}. Hence, we replace the objective function \eqref{first_obj} with its expectation equivalent as follows:
\begin{equation}
\allowdisplaybreaks
     \label{master_full_obj} {\mcal{Q}}_o := \min \sum_{i \in \mcal{N}^g} z^g_ic^0_i +  c_i^1\rho_i + c_i^2(\rho_i)^2  + \sup_{\bb{P} \in \bb{Q}}   \bb{E_P}[\mcal{H}(\bsy{z^a},\bsy{z^g}, \bsy{\rho}, \bsy{\widetilde{\nu}^d})]
\end{equation}
Where $\sup_{\bb{P} \in \bb{Q}}   \bb{E_P}[\mcal{H}(\bsy{z^a},\bsy{z^g}, \bsy{\rho}, \bsy{\widetilde{\nu}^d})]$ is the worst-case expected operating cost incurred in the second stage over an ambiguity set $\mathbb{Q}$ of an unknown probability measure $\mathbb{P}$ of random variables $\wt{\nu}^N$ and $\wt{\nu}^E$, which we discuss in detail in Section \ref{sec:Model of Uncertainty}. 
}

{\color{Black}
\subsection{Convex Relaxations}
{\label{OA}}


The two-stage DR-OTSGMD model is a very hard mixed-integer non-convex problem because of constraints \eqref{z_nl_pij}-\eqref{z_nl_qji}, \eqref{nl_pbalance_slack}, \eqref{nl_qbalance_slack}, \eqref{qloss}, and \eqref{dc_ij_zag}-\eqref{nl_qbalance_slack_qloss}. One of the focuses of this paper is 
the derivation of lower bounds to DR-OTSGMD, so we relax each of these constraints. For notation brevity, in the following sections, $L_x$ denotes a set of uniformly located points within the bounds of the variable $x$.

\paragraph{Quadratic Fuel Cost Function.} In this paper, the fuel cost is formulated as a quadratic function of the real power generation from generator $i \in \mcal{G}$ (i.e., $c^1_i\rho_i + c^2_i(\rho_i)^2 $). By applying the perspective reformulation described in \cite{gunluk2010perspective, ostrowski2012tight,frangioni2008tighter}, the quadratic term $c^2_i(\rho_i)^2$ is reformulated as:
%
\begin{equation}
\allowdisplaybreaks
\label{perspective_cut}
     z^g_i\widecheck{\rho}_i \geq \rho_i^2 \qquad \forall i \in \mcal{G}
\end{equation}%

\noindent 
Constraints (\ref{perspective_cut}) guarantees that the minimum (optimal) value of $\widecheck{\rho}_i$ is zero if generator $i$ is switched off (i.e., $z_i^g = 0$). Otherwise $\widecheck{\rho}_i$ is equal to $\rho_i^2$ when $z_i^g = 1$ because this is a minimization problem. Next, we replace constraints (\ref{perspective_cut}) with a set of piecewise linear constraints which outer approximate (relax) $\widecheck{\rho}_i$ as follows:
\begin{equation}
    \label{perspective_cut_l} \widecheck{\rho}_i \geq 2\rho_i{\ell} - z_i^g(\ell)^2 \qquad \forall i \in \mcal{N}^g, \ell \in L_{\rho_i} 
\end{equation}
Constraints \eqref{perspective_cut_l} describe a set of linear inequalities that are tangent to the quadratic curve $\rho^2_i$ at points $l \in L_{\rho_i}$. As a result, the objective function is modified as follows.
\begin{equation}
    \label{master_full_obj_l} \min \sum_{i \in \mcal{N}^g} z^g_ic^0_i +   c_i^1\rho_i + c_i^2\widecheck{\rho}_i  + \sup_{\bb{P} \in \bb{Q}} \bb{E_P}[\mcal{H}(\bsy{z^a},\bsy{z^g}, \bsy{\rho}, \bsy{\widetilde{\nu}^d})] 
\end{equation}
Similarly, we use the same method to relax the quadratic term $\Delta_i^2$ and replace equations \eqref{second} with equations \eqref{second_obj_L}.
\begin{subequations}
\label{second_obj_L}
\begin{align}
    & \label{second_obj_l} \mcal{H}(\bsy{z^a}, \bsy{z^g}, \bsy{\rho}, \bsy{\widetilde{\nu}^d}) = \min \sum_{i \in \mcal{G}} c_i^{R2}\widecheck{\Delta}_i +  c_i^{R1}\Delta_i + \sum_{i \in \mcal{N}^a} \kappa_i (l^{p+}_{i}+l^{q+}_{i} + l^{p-}_{i}+l^{q-}_{i}) \\
    &\label{wc_delta_l} \widecheck{\Delta}_i \geq 2{\ell}(\Delta_i) - z_i^g({\ell})^2 \qquad \forall i \in \mcal{N}^g, \ell \in L_{\Delta_i}
\end{align}
\end{subequations}

\paragraph{AC Power Flow Physics.} In constraints \eqref{z_nl_pij}-\eqref{z_nl_qji}, \eqref{nl_pbalance_slack} and \eqref{nl_qbalance_slack_qloss}, the nonlinearities are expressed with the following terms \cite{hijazi2017convex}:
\begin{subequations}
\small
\allowdisplaybreaks
\begin{align}
    &\nonumber (1)\ w_{i} := v_i^2, \ \ (2) \  w^z_{ie} := z_ev_i^2 \ (w^z_{je} := z_ev_j^2), \ \ (3) \  w_e^c := z_ev_iv_j\cos(\theta_i-\theta_j), \ \ (4) \  w_e^s :=z_ev_iv_j\sin(\theta_i-\theta_j)
\end{align}
\end{subequations}%
The auxiliary variables $w_i$, $w^z_{ie}$, $w^z_{je}$, $w_e^c$ and $w_e^s$ are introduced to relax each of these quantities. First, the quadratic term $\ w_i := v_i^2$, is convexified by applying the second-order conic (SOC) relaxation defined in equation \eqref{eq:vi_soc}.
\begin{subequations}
\label{eq:vi_soc}
\begin{align}
& \label{vi_soc_l} w_i \geq v^2_i \qquad \forall i \in \mathcal{N}^a \\
& \label{vi_soc_u} w_i \leq (\overline{v}_i + \underline{v}_i)v_i - \overline{v}_i\underline{v}_i \qquad \forall i \in \mathcal{N}^a 
\end{align}
\end{subequations}
\noindent Further, we outer approximate (relax) the convex envelope of $v^2_i$ using a set of piecewise linear constraints which replace constraints \eqref{vi_soc_l} with constraints \eqref{vi_soc_l_oa}:
\begin{equation}
\label{vi_soc_l_oa} w_i \geq 2{\ell}v_i - ({\ell})^2 \qquad \forall i \in \mcal{N}^a, \ \ell \in L_{v_i}
\end{equation}

Next, to relax non-convex constraints $w^z_{ie}:= z_ev_i^2$ and $w^z_{je}:= z_ev_j^2$, we introduce the following notation: Given any two variables $x_i$, $x_j \in \mathbb{R}$, the McCormick (MC) relaxation \cite{mccormick1976computability} is used to relax a bilinear product $x_i\cdot x_j$ by introducing an auxiliary variable ${x}_{ij} \in {\langle x_i, x_j \rangle}^{MC}$. The feasible region of ${x}_{ij}$ is given by:
\vspace{-0.5cm}
\begin{subequations} \label{eq:SMC}
\allowdisplaybreaks
\begin{align}
& \label{McCormick}{x}_{ij} \geq \underline{x}_ix_j+\underline{x}_jx_i -\underline{x}_i \hspace{2pt} \underline{x}_j \\ 
&{x}_{ij} \geq \overline{x}_ix_j+\overline{x}_jx_i - \overline{x}_i \hspace{2pt} \overline{x}_j \\ 
&{x}_{ij} \leq \underline{x}_ix_j+\overline{x}_jx_i-\underline{x}_i \hspace{2pt} \overline{x}_j \\ 
&{x}_{ij} \leq \overline{x}_ix_j+\underline{x}_jx_i-\overline{x}_i \hspace{2pt} \underline{x}_j \\
&\underline{x}_i \leq x_i \leq \overline{x}_i, \ \underline{x}_j \leq x_j \leq \overline{x}_j
\end{align}
\end{subequations}%
Applying the relaxations in \eqref{eq:vi_soc} to $v_i^2$ and $v_j^2$, results in auxiliary variables, $w_i$ and $w_j$, respectively. Thus, the constraints $w^z_{ie}:= z_ev_i^2$ and $w^z_{je}:= z_ev_j^2$ can be expressed as $w^z_{ie}:= \langle z_e, w_i \rangle^{MC}$ and $w^z_{je}:= \langle z_e, w_j \rangle^{MC}$, respectively. It is also important to notice that the MC relaxation is exact when one variable in the product is binary \cite{nagarajan2016tightening}.

Finally, to deal with nonlinear terms $w_e^c := z_ev_iv_j\cos(\theta_i-\theta_j)$ and $w_e^s :=z_ev_iv_j\sin(\theta_i-\theta_j)$, though there are numerous tight convex relaxations in the literature \cite{molzahn2019survey,coffrin2016qc,lu2018tight,Narimani2018,sundar2018optimization}, we apply state-of-the-art SOC relaxations of \cite{kocuk2017new} considering it's computational effectiveness, as formulated in constraints \eqref{eq:we_soc}. 
\begin{subequations}
\label{eq:we_soc}
\begin{align}
& \label{we_z} z_e\underline{w}^c_e \leq w^c_e \leq z_e\overline{w}^c_e, \ z_e\underline{w}^s_e \leq w^s_e \leq z_e\overline{w}^s_e \qquad  \forall e_{ij}  \in \mcal{E}^a \\
&\label{we_link_1} \tan(\underline{\theta})w^c_e \leq w^s_e \leq \tan(\overline{\theta})w^c_e \qquad \forall e_{ij} \in \mathcal{E}^a \\
&\label{we_link_2} (w^c_e)^2 + (w^s_e)^2 \leq w^z_{ie}w^z_{je} \qquad \forall e_{ij} \in \mathcal{E}^a
\end{align}
\end{subequations}
\noindent Note that constraint set \eqref{we_link_1} is equivalent to the phase angle limit constraints \eqref{thetaij_ub}. Constraint set \eqref{we_link_2} is the SOC relaxations of equation $z^a_e\big((w^c_e)^2 + (w^s_e)^2 - w_iw_j\big) = 0$ which is obtained by linking $w^c_e$ and $w^s_e$ through the trigonometric identity $\cos^2(\theta_i-\theta_j) + \sin^2(\theta_i-\theta_j) = 1$. Further, we use a set of piecewise linear constraints to outer approximate (relax) the rotated SOC constraints \eqref{we_link_2} such that:
\begin{eqnarray}
\allowdisplaybreaks
\label{we_link_l}
     &\label{we_link_l_OA} 2\Big(w_e^c{\ell^c} + w_e^s{\ell^s} \Big) - w^z_{ie}\ell^{t} - w^z_{je}\ell^{f} \leq z_e\big(({\ell^c})^2 + ({\ell^s})^2  - \ell^{f}\ell^{t}\big) \quad \forall e_{ij} \in \mcal{E}^a \nonumber \\
    & \hspace{2.3cm}  \ell^c \in L_{{w^c_e}},
    \ell^s \in L_{w^s_e}, \ell^f \in L_{w_i}, \ell^t \in L_{w_j}
\end{eqnarray}
%

Based on these relaxations, we replace the non-convex
constraints in \eqref{z_nl_pij}-\eqref{z_nl_qji}, \eqref{nl_pbalance_slack} and \eqref{nl_qbalance_slack_qloss} with constraints \eqref{eq:ac_soc}. Constraints (\ref{pij_l})--(\ref{qji_l}) force line flows to be zero when the line is switched-off and take the associated values otherwise.
%
\begin{subequations}
\allowdisplaybreaks
\label{eq:ac_soc}
\begin{align}
&\label{pbalance_l} \smashoperator{\sum_{e_{ij} \in \mcal{E}_i^+} } p_{ij} +
\smashoperator{\sum_{e_{ji} \in \mcal{E}_i^-} } p_{ij}
= \sum_{k \in \mathcal{G}_i}f^p_{k}+l^{p+}_{i}-l^{p-}_{i}-d^p_{i} - w_ig_i \qquad \forall i\in \mcal{N}^a\\
&\label{pbalance_l} \smashoperator{\sum_{e_{ij}\in \mcal{E}^+_i}} q_{ij} +\smashoperator{\sum_{e_{ji}\in \mcal{E}^-_i}} q_{ij} 
= \sum_{k \in \mathcal{G}_i}f^q_{k}+l_i^{q+}-l^{q-}_{i}-d^q_{i} - d_{i}^{qloss} + w_ib_i \qquad \forall i \in \mcal{N}^a\\
& \label{pij_l} p_{ij}=g_{e}w^z_{ie}-g_{e}w^c_e - b_{e}w^s_e \qquad \forall e_{ij} \in \mcal{E}^a \\
& \label{qij_l}  q_{ij}=-(b_{e}+\frac{b_{e}^c}{2})w^z_{ie} + b_{e}w^c_e - g_{e}w^s_e \qquad \forall e_{ij} \in \mcal{E}^a\\
& \label{pji_l}  p_{ji}=g_{e}w_{je}^z-g_{e}w^c_e + b_{e}w^s_e \qquad \forall e_{ij} \in \mcal{E}^a \\
& \label{qji_l} q_{ji}=-(b_{e}+\frac{b_{e}^c}{2})w_{je}^z + b_{e}w^c_e + g_{e}w^s_e \qquad \forall e_{ij} \in \mcal{E}^a\\ 
& \label{zw_mc}w_{ie}^z := \langle z_e, w_i \rangle^{MC}, \ \ w_{je}^z := \langle z_e, w_j \rangle^{MC} \qquad \forall e_{ij} \in \mcal{E}^a \\
& \eqref{vi_soc_u}, \eqref{vi_soc_l_oa}, \eqref{we_z}-\eqref{we_link_1}, \eqref{we_link_l_OA}
\end{align}
\end{subequations}%
\paragraph{Line Thermal Limits.} Similar to constraints \eqref{vi_soc_l_oa}, we replace constraints \eqref{z_nl_capacity} with \eqref{eq:thermal_l} which outer approximate (relax) quadratic terms $p^2_{ij}+q^2_{ij}$ and $p^2_{ji}+q^2_{ji}$.
\begin{subequations}
\allowdisplaybreaks
\label{eq:thermal_l}
\begin{align}
    & 2p_{ij}{\ell}^p + 2q_{ij}{\ell^q} \leq z^a_e\left(s^2_e + ({\ell}^p)^2 + ({\ell}^q)^2\right) \qquad \forall e_{ij} \in \mcal{E}^a, \ \ell^p \in L_{p_e}, \ \ell^q \in L_{q_e} \\
    & 2p_{ji}{\ell^p}+ 2q_{ji}{\ell^q} \leq z^a_e\left(s^2_e + ({\ell}^p)^2 + ({\ell}^q)^2\right) \qquad \forall e_{ij} \in \mcal{E}^a, \ \ell^p \in L_{p_e}, \ \ell^q \in L_{q_e} 
\end{align}
\end{subequations}%

\paragraph{GIC-Associated Effects.} In the second stage, given geo-electric fields induced by a realized GMD, all nonlinearities and non-convexities are in the bilinear form as: (1) $z^a_{\overrightarrow{e}}(v_m^d-v_n^d)$ , 
and (2) $v_iI_e^d$. Using MC relaxations as described in \eqref{eq:SMC}, we replace constraints \eqref{qloss} with \eqref{qloss_l}-\eqref{mc_vi} and replace constraints \eqref{dc_ij_zag} with \eqref{dc_ij_l}-\eqref{mc_zgv}. 
\begin{subequations}
\label{eq:dc_l}
\begin{align}
& \label{qloss_l} d_{i}^{qloss}=\smashoperator{ \sum_{e \in \mcal{E}_i^{\tau}}}k_{e}u_{ie}^d \qquad \forall i \in \mcal{N}^a\\
&\label{mc_vi} u_{ie}^{d} \in {\langle v_i, \ \wh{I}^d_e \rangle}^{MC} \qquad \forall i \in \mcal{N}^a, \ e \in \mcal{E}_i^{\tau} \\
& \label{dc_ij_l} I_e^d = a_ev_{mn}^{zd} + z^a_{\overrightarrow{e}}a_e\wt{\nu}^d_e \qquad \forall e_{mn} \in \mcal{E}^d \\
&\label{mc_zgv} v_{mn}^{zd} \in {\langle z^g_{\overrightarrow{i}}, \ v_m^d-v_n^d \rangle}^{MC} \qquad \forall e_{mn} \in \mcal{E}^d 
\end{align}
\end{subequations}%
where $v_{mn}^{zd}$ and $u^d_{ie}$ are introduced as $v_{mn}^{zd} := z^a_{\overrightarrow{e}}(v^d_m-v^d_n)$ and $u^d_{ie} := v_iI^d_e$.
}
{
\color{red}

\label{sec:formulation}

}

\subsection{Model of Uncertainty} 
\label{sec:Model of Uncertainty}
This section focuses on constructing an ambiguity set (denoted by $\mathbb{Q}$) of probability distributions for the random geo-electric field induced by a GMD event. The notation $\bsy{\widetilde{\omega}}$ 
is used to denote the
vector of all random variables (i.e., $\bsy{\widetilde{\omega}} = [\widetilde{\nu}^E, \widetilde{\nu}^N]^T$) and the notation $\bsy{\mu} = [\mu^E, \mu^N]^T$ is used to denote the mean vector of the eastward  and  northward geo-electric fields. The ambiguity set $\mathbb{Q}$ is then formulated with: 

\begin{equation}
\label{ambi_1}
  \bb{Q}=\begin{cases}
    & \widetilde{\bsy{\omega}} \in \bb{R}^2\\
    \bb{P} \in \mcal{P}_0(\bb{R}):& \bb{E}_{\bb{P}}(\widetilde{\bsy{\omega}}) = {\bsy{\mu}}\\
    & \bb{P}\Big(\widetilde{\bsy{\omega}} \in \Omega\Big) = 1
  \end{cases}
\end{equation}%

\noindent where $\mcal{P}_0(\bb{R})$ denotes the set of all probability distributions on $\bb{R}$ and $\bb{P}$ denotes a probability measure in $\mcal{P}_0$. It is important to note that random variables $\bsy{\widetilde{\omega}}$  are not associated with any specific probability distribution. The second line of equation \eqref{ambi_1} forces the expectation of $\widetilde{\bsy{\omega}}$ to be $\bsy{\mu}$. The third line defines the support of the random variables, $\Omega$, and contains all the possible outcomes of $\bsy{\widetilde{\omega}}$. We formulate this support set based on valid bounds of geo-electric fields in eastward and northward directions, such that:
%

\vspace{-0.5cm}
\begin{equation}
\allowdisplaybreaks
\label{eq:support}
     \Omega := \Big\{ (\widetilde{\nu}^N, \widetilde{\nu}^E) \in \bb{R}:  0 \leq \widetilde{\nu}^N \leq \overline{\nu}^M, \ -\overline{\nu}^M \leq \widetilde{\nu}^E \leq \overline{\nu}^M, \ (\widetilde{\nu}^N)^2 + (\widetilde{\nu}^E)^2 \leq
    (\overline{\nu}^M)^2 
    \Big\}
\end{equation}%
\noindent As introduced in Section \ref{subsubsection:geo-electric field calc}, the geo-electric field $\overrightarrow{E}$ is formulated as a vector and represented by its two components $\widetilde{\nu}^E$ and $\widetilde{\nu}^N$. Here, we assume the direction of $\overrightarrow{E}$ is orientated from $0\degree$ to $180\degree$. The first two constraints in $\Omega$ define individual bounds of $\widetilde{\nu}^N$ and $\widetilde{\nu}^E$, respectively. {\color{Black}Given a prespecified upper bound on the geo-electric field amplitude (denoted by $\overline{\nu}^M$), the third inequality further limits the feasible values of $\widetilde{\nu}^N$ and $\widetilde{\nu}^E$ which is formulated as a quadratic constraint and inferred by the geometric representation of $\overrightarrow{E}$. To better demonstrate set $\Omega$, figure \ref{Fig:support} gives two examples that illustrate feasible domains of $\widetilde{\nu}^E$ and $\widetilde{\nu}^N$, where vector $\vec{E}$ represents the geo-electric field in the area of a transmission system. In the figure, the half circle represents the quadratic boundary of $\vec{E}$ (i.e., $(\widetilde{\nu}^E)^2 + (\widetilde{\nu}^N)^2 \leq (\overline{\nu}^M)^2$)}

\begin{figure}[h]
\captionsetup{font=small}
  \centering
  \subfigure[$\widetilde{\nu}^E < 0$, $|\vec{E}| = \overline{\nu}^M$]{
  \includegraphics[scale=0.47]{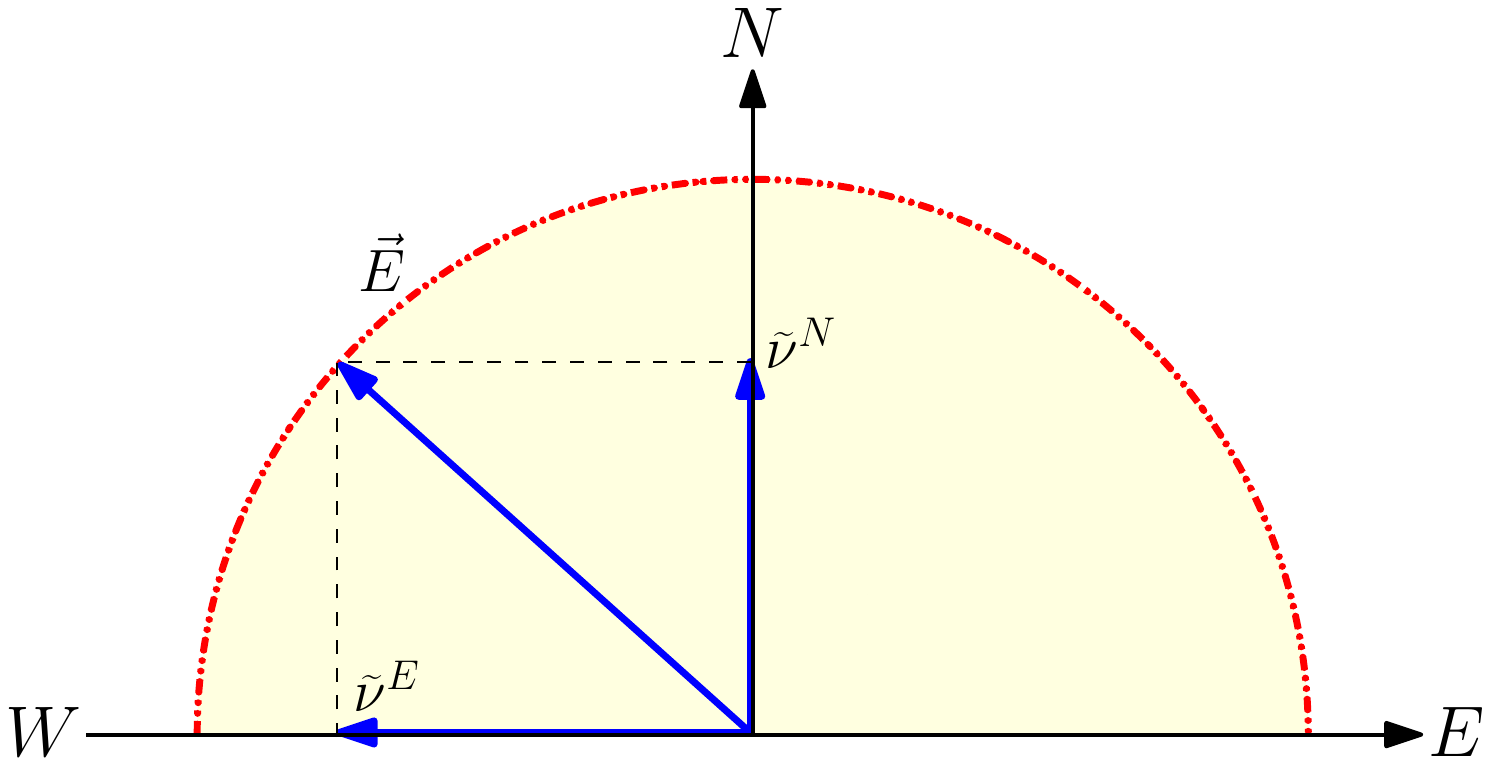}
  \label{fig:vd_WN}
  }
  \subfigure[$\widetilde{\nu}^E > 0$, $|\vec{E}| < \overline{\nu}^M$]{
  \includegraphics[scale=0.47]{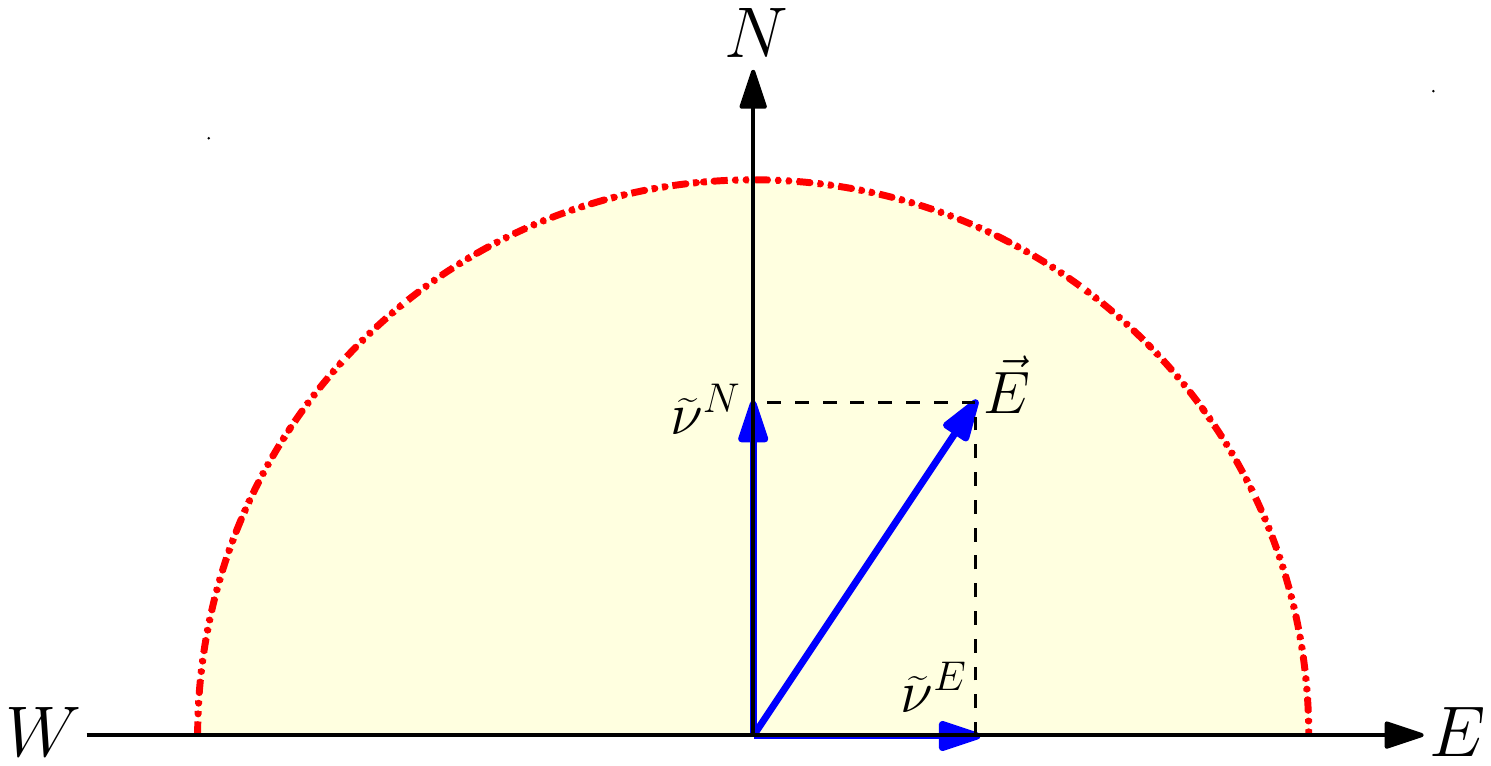}
  \label{fig:vd_EN}
  } 
\caption[Examples of the feasible regions of northward and eastward geo-electric fields]{Examples of feasible regions of $\widetilde{\nu}^E$ and $\widetilde{\nu}^N$. The geo-electric field $\vec{E}$ is formed by $\widetilde{\nu}^E$ and $\widetilde{\nu}^N$. In Fig.  \ref{fig:vd_WN}, the magnitude of $\vec{E}$ equals $\overline{\nu}^M$ and angle relative to east is larger than $\frac{\pi}{2}$. In Fig. \ref{fig:vd_EN}, the magnitude of $\vec{E}$ is less than $\overline{\nu}^M$ and the angle is smaller than $\frac{\pi}{2}$.} 
\label{Fig:support}
\end{figure}

\subsection{Reformulation of the Two-Stage DR-OTSGMD}

Using the relaxation methods described in Section \ref{OA}, the subproblem is linear and is a \textit{lower bound} of the original objective. Adopting the notation of \cite{zhao2018distributionally}, we present the linearized program in a succinct form as follows:
\begin{subequations}
\allowdisplaybreaks
\label{eq:suc}
\begin{align}
    &\label{suc_objective} \min_{\bsy{y}} \bsy{a}^T\bsy{y} + \sup_{\bb{P} \in \bb{Q}}\bb{E}_{\bb{P}}[\mcal{H}(\bsy{y}, \widetilde{\bsy{\omega}})] \\
    &\label{suc_0} \qquad \quad \bsy{Ay} \geq \bsy{b}  \\
    & \text{where} \nonumber  \\
    &\qquad \quad \mcal{H}(\bsy{y}, \widetilde{\bsy{\omega}}) = \min_{\bsy{x}} \bsy{c}^T\bsy{x} \\
    &\qquad \quad \label{suc_1} \quad \textit{s.t.} \ \
    \bsy{Gy} + \bsy{Ex} \geq \bsy{h} \\
    & \qquad\qquad\qquad \label{suc_2}\bsy{T(\widetilde{\omega})}\bsy{y} = \bsy{Wx}
\end{align}
\end{subequations}%
{\color{Black}
\noindent where $\bsy{y}$ denotes the first-stage variables and $\bsy{x}$ denotes the second-stage recourse variables that depend on random variables $\bsy{\widetilde{\omega}} = [\widetilde{\nu}^E, \widetilde{\nu}^N]^T$. Constraint set (\ref{suc_0}) represents constraints \eqref{setpoint} and \eqref{perspective_cut_l} in standard form. Constraint set (\ref{suc_1}) represents constraints \eqref{vi}, \eqref{gp_za}-\eqref{gq_za}, \eqref{ramp}, \eqref{dc_balance}, \eqref{eq:effective_GIC}, \eqref{ramp_up}, \eqref{eq:ac_soc}-\eqref{eq:dc_l}. Constraint set (\ref{suc_2}) represents constraints \eqref{eq:voltage source} and \eqref{dc_ij_l} that are 
the equations that include
random variables in the problem. Here, $\bsy{T(\widetilde{\omega})}$ is an affine function of $\bsy{\widetilde{\omega}}$ such that $\bsy{T(\widetilde{\omega})} = \sum^{|\bsy{\widetilde{\omega}}|}_{i=1}\bsy{T}^i\widetilde{\omega}_i$.} In the rest of this section, we derive a reformulation of the two-stage DR-OTSGMD problem which can be solved by applying a decomposition framework.
{\color{Black}
\subsubsection{Handling the worst-case expected cost}
According to Zhao and Jiang \cite{zhao2018distributionally}, an equivalent problem to \eqref{eq:suc} rewrites the worst-case expected cost $\sup_{\bb{P}\in \bb{Q}}\bb{E}_{\bb{P}}[\mcal{H}(\bsy{y}, \widetilde{\bsy{\omega}})]$ in the objective function (\ref{suc_objective}) as problem (\ref{eq:Q1}).}
%
\begin{subequations}
\allowdisplaybreaks
\label{eq:Q1}
\begin{align}
    & \sup_{\bb{P} \in \bb{Q}} \bb{E}_{\bb{P}}[\mcal{H}(\bsy{y}, \widetilde{\bsy{\omega}})] = \max_{\bb{P}}\int_{\Omega} \mcal{H}(\bsy{y}, \widetilde{\bsy{\omega}})\bb{P}(d\widetilde{\bsy{\omega}}) \\
    & \label{Q1_mean} \quad \textit{s.t.} \ \ \int_{\Omega}\widetilde{\bsy{\omega}} \ \bb{P}(d\widetilde{\bsy{\omega}}) = \bsy{\mu} \\
    & \label{Q1_support} \qquad \quad \int_{\Omega}\bb{P}(d\widetilde{\bsy{\omega}}) = 1
\end{align}
\end{subequations}

\noindent {\color{Black} Here, constraints (\ref{Q1_mean})--(\ref{Q1_support}) precisely describe the ambiguity set defined in (\ref{ambi_1}). Constraint (\ref{Q1_mean}) defines the mean value of random variables $\widetilde{\bsy{\omega}}$. Equation \eqref{Q1_support} formulates the support set $\Omega$ which contains all the possible outcomes of $\widetilde{\bsy{\omega}}$. Then, based on duality theory \cite{bertsimas1997introduction}, the problem is equivalently reformulated as \eqref{eq:rf_0} by incorporating the dual problem of \eqref{eq:Q1} in formulation \eqref{eq:suc}.}
%
%
\begin{subequations}
\allowdisplaybreaks
\label{eq:rf_0}
\begin{align}
    &\label{master_obj} {\mcal{Q}}_d := \min_{\bsy{y}, \bsy{\lambda}, \eta} \bsy{a}^T\bsy{y} + \bsy{\mu}^T{\bsy{\lambda}} + \eta \\
    &\quad \textit{s.t.} \ \ \eqref{suc_0} \nonumber \\
    &\qquad \quad \label{optimality_cut} \eta \geq \mcal{H}(\bsy{y}, \widetilde{\bsy{\omega}}) - {\bsy{\lambda}^T}\widetilde{\bsy{\omega}} \qquad \forall \widetilde{\bsy{\omega}} \in \Omega \\
    & \text{where} \nonumber \\
    &\label{sub}\qquad \quad \mcal{H}(\bsy{y}, \bsy{\widetilde{\omega}}) = \min_{\bsy{x}}\big\{\bsy{c}^T\bsy{x}: \bsy{Gy} + \bsy{Ex} \geq \bsy{h}, \  \bsy{T(\widetilde{\omega})y} = \bsy{Wx}\big\}
\end{align}
\end{subequations}%

\noindent{\color{Black} where $\bsy{\lambda}$ and $\eta$ are dual variables associated with constraints ($\ref{Q1_mean}$) and ($\ref{Q1_support}$), respectively. In addition, since constraints (\ref{optimality_cut}) must be satisfied for all realizations of $\bsy{\widetilde{\omega}}$ over support $\Omega$,  the equivalent formulation of constraints (\ref{optimality_cut}) is expressed as:} 
%
\begin{equation}
     \label{eq:inner_max}\eta \geq \max_{\forall \widetilde{\bsy{\omega}} \in \Omega}\big\{ \mcal{H}(\bsy{y}, \widetilde{\bsy{\omega}}) - {\bsy{\lambda}^T}\widetilde{\bsy{\omega}} \big\} 
\end{equation}
\noindent 
Further, by applying standard duality theory, the right-hand side of constraint (\ref{eq:inner_max}) is rewritten as the inner maximization problem:
\begin{subequations}
\allowdisplaybreaks
\label{eq:dual_sp}
\begin{align}
    &\label{inner_max} \max_{\forall \widetilde{\bsy{\omega}} \in \Omega}\Big\{ \mcal{H}(\bsy{y}, \widetilde{\bsy{\omega}}) - {\bsy{\lambda}^T}\widetilde{\bsy{\omega}} \Big\} \ \ = \max_{\forall (\bsy{\gamma}, \bsy{\pi}) \in \bsy{\Gamma}, \bsy{\widetilde{\omega}} \in \Omega}\Big\{ (\bsy{h}-\bsy{Gy})^T\bsy{\gamma} + \bsy{y}^T\big(\sum^{|\bsy{\widetilde{\omega}}|}_{i=1}(\bsy{T}^i)^T\bsy{\pi}\widetilde{\omega}_i\big) - \bsy{\lambda}^T\bsy{\widetilde{\omega}} \Big\} \\
    &\label{sp_feas_dual} \bsy{\Gamma} = \left\{
        \begin{array}{ll}
        \bsy{E}^T\bsy{\gamma}  + \bsy{W}^T\bsy{\pi} \leq \bsy{c} \\
        \bsy{\gamma} \geq \bsy{0} \\
        \end{array}
          \right.
\end{align}
\end{subequations}

\noindent {\color{Black}where $\bsy{\gamma}$ and $\bsy{\pi}$} are dual variables associated with constraints (\ref{suc_1}) and (\ref{suc_2}), respectively. $\bsy{\Gamma}$ denotes the feasible domain of the dual variables. 

{\color{Black}
\subsubsection{Relaxing bilinear terms in the reformulation}
}
For fixed $\bsy{y}$ and $\bsy{\lambda}$, the objective function (\ref{inner_max}) contains bilinear terms $\bsy{\pi}\widetilde{\omega}_i$. The literature presents several methods to solve robust optimization (RO) models with bilinear terms in the inner (sub-) problem, including heuristics \cite{wu2017adaptive} and mixed-integer programming (MIP) reformulations \cite{minguez2016robust, zeng2013solving, wang2014robust}. The former finds local optima for two-stage robust problems for which the uncertainty set is a general polyhedron (e.g., non-convex regular polyhedra \cite{grunbaum1977regular}). The latter finds exact solutions for RO problems by exploiting the special structure of the uncertainty set. A recent work by Zhao and Zeng \cite{zhao2012robust} develops an exact algorithm for two-stage RO problems for which the uncertainty set is a general polyhedron. The authors derive an equivalent MIP formulation for the inner problem by using strong Karush–Kuhn–Tucker (KKT) conditions. However, this MIP formulation is challenging to solve when binary variables are introduced to linearize complementary slackness constraints. Hence, it is difficult to derive exact solutions for RO models for which the uncertainty set is a general polyhedron. Thus, we leave the exploration of exact solution methods as future work. In this research, we focus on generating a lower bound for the reformulation of $\mcal{Q}^d$ described in (\ref{eq:rf_0}).

From formulation (\ref{eq:dual_sp}), we observe that $\bsy{\widetilde{\omega}}$ only appears in the objective function (\ref{inner_max}), thus different realizations of $\bsy{\widetilde{\omega}}$ do not affect the feasible region of the dual variables, $\bsy{\Gamma}$ \eqref{sp_feas_dual}. We note that $\bsy{\Gamma}$ is nonempty, since formulation (\ref{sub}) is feasible for any given first-stage solution due to load shedding and over-generation options. Let $(\wh{\bsy{\pi}}, \wh{\bsy{\gamma}})\in \bsy{\Gamma}$ be any feasible dual solution to formulation (\ref{eq:dual_sp}). It follows that $\bsy{\wh{y}}$ and $\bsy{\wh{\lambda}}$ are solutions of first-stage variables $\bsy{y}$ and $\bsy{\lambda}$, respectively. {\color{Black}Given a solution $\bsy{\wh{y}}$, $\bsy{\wh{\lambda}}$, $\bsy{\wh{\pi}}$ and $\bsy{\wh{\gamma}}$, the maximum value of (\ref{inner_max}) is the optimal objective value of \eqref{fixed_2DRO} where ${\widetilde{\bsy{\omega}}}$ are the only decision variables:}
\begin{equation}
\label{fixed_2DRO}
    \max_{\forall \bsy{\widetilde{\omega}} \in \Omega}\Big\{ (\bsy{h}-\bsy{G\wh{y}})^T\bsy{\wh{\gamma}} + \bsy{\wh{y}}^T\big(\sum^{|\bsy{\widetilde{\omega}}|}_{i=1}(\bsy{T}^i)^T\bsy{\wh{\pi}}\widetilde{\omega}_i\big) - \bsy{\wh{\lambda}}^T\bsy{\widetilde{\omega}} \Big\}
\end{equation}
%

\noindent {\color{Black}Based on LP geometry \cite{bertsimas1997introduction, boyd2004convex}, if $\Omega$ is a bounded and nonempty polyhedron, then there exists an optimal (worst-case) solution of $\bsy{\wt{\omega}}$ which is an extreme point of $\Omega$. However, as described in Eq.~\eqref{eq:support}, $\Omega$ is convex set bounded by a quadratic curve that has an infinite number of extreme points. Thus, we discretize the curve and create a polyhedral support set $\Omega'$ that is bounded by a piecewise linear relaxation of the curve. Hence, we modify formulation \eqref{fixed_2DRO} to be \eqref{fixed_2DRO_relaxed}. Under this formulation there must be a vertex of $\Omega'$ which is an optimal $\bsy{\wt{\omega}}$ for \eqref{fixed_2DRO_relaxed}.
\begin{equation}
\label{fixed_2DRO_relaxed}
    \max_{\forall \bsy{\widetilde{\omega}} \in \Omega'}\Big\{ (\bsy{h}-\bsy{G\wh{y}})^T\bsy{\wh{\gamma}} + \bsy{\wh{y}}^T\big(\sum^{|\bsy{\widetilde{\omega}}|}_{i=1}(\bsy{T}^i)^T\bsy{\wh{\pi}}\widetilde{\omega}_i\big) - \bsy{\wh{\lambda}}^T\bsy{\widetilde{\omega}} \Big\}
\end{equation}
} 
\noindent We formalize the properties of $\Omega'$ with Lemma \ref{prop_1}.
%
\begin{lem}\label{prop_1}
{\color{Black}Given that $\Omega'$ is a linear relaxation of $\Omega$ and $\Omega' \in \Omega$.} Then, for any given first-stage decisions $\bsy{y}$ and $\bsy{\lambda}$, {\color{Black}there exists an extreme point of $\Omega'$ which is an optimal (worst-case) solution of $\bsy{\wt{\omega}} \in \Omega'$}.
\end{lem}
{\color{Black}
\begin{proof}
Given that $\Omega$ is bounded and nonempty as defined in \eqref{eq:support} and $\Omega$ is part of a max function in a minimization problem, $\Omega'$ is a linear relaxation of $\Omega$. As illustrated in Figure \ref{fig:Omega_relaxed}, $\Omega'$ is a nonempty polyhedron that is bounded by linear cuts going through a series of points on the curve (i.e., the boundary of $\Omega$). Then, for fixed decisions $\bsy{\wh{y}}$ and $\bsy{\wh{\lambda}}$, the problem \eqref{fixed_2DRO_relaxed} is a linear programming problem of finding the worst-case $\bsy{\wh{\omega}}$ over a nonempty and bounded polyhedron $\Omega'$. Thus, based on LP geometry \cite{bertsimas1997introduction, boyd2004convex}, there exists an extreme point of $\Omega'$ which is optimal.
\end{proof}. 
}
\begin{figure}[h]
    \centering
    \includegraphics[scale=0.51]{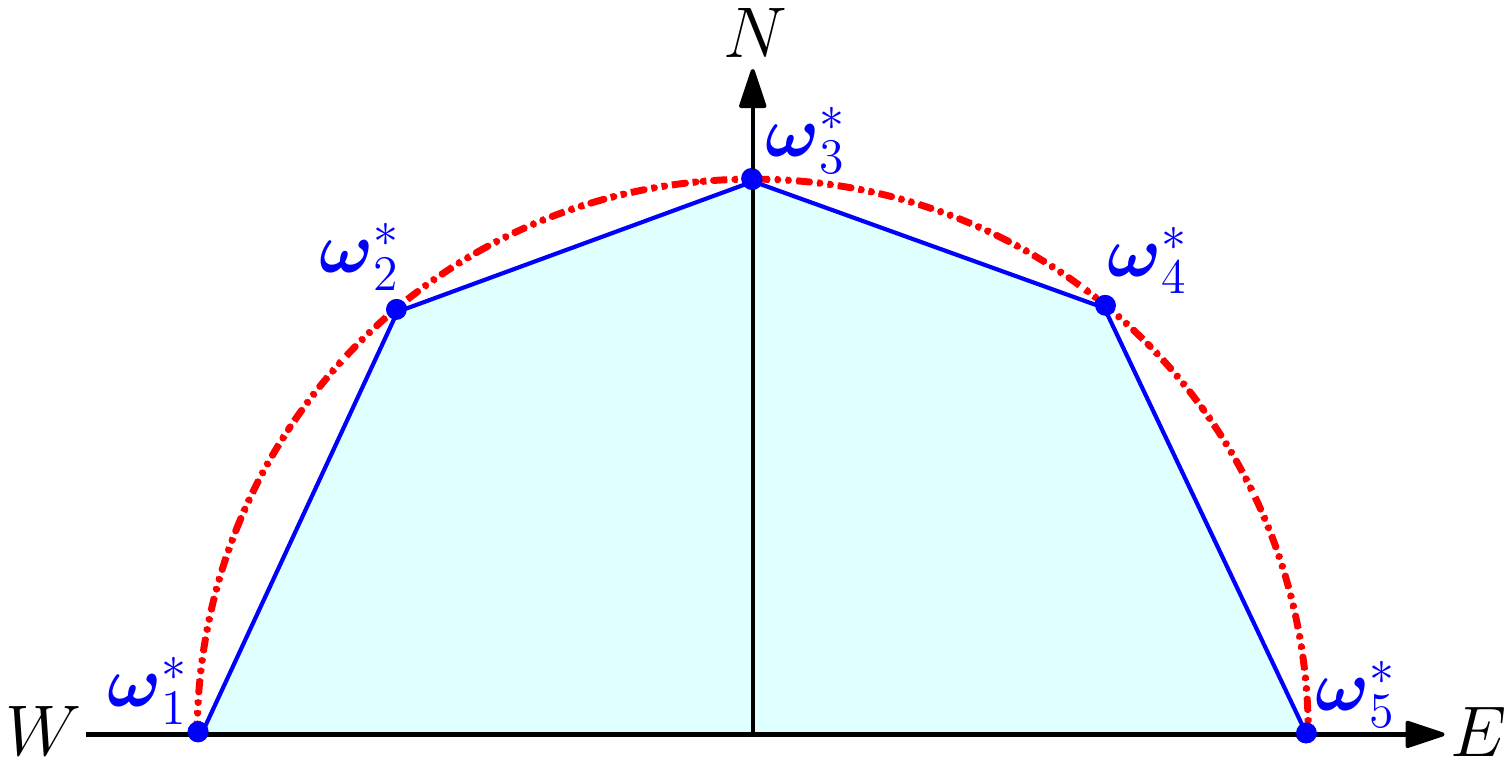}
    \caption{\color{Black}A linear relaxation of support set $\Omega$. The blue shadow represents $\Omega'$ and $\bsy{\omega}^*_i$ for $i \in \{1,2,..,5\}$ denote the vertices of $\Omega'$.}
    \label{fig:Omega_relaxed}
\end{figure}

{\color{Black} 
\subsubsection{A MIP reformulation of the two-stage DRO-OTSGMD problem}

{\color{black}
Based on proposition \ref{prop_1}, we derive a MIP formulation (denoted as $\mcal{Q}_v$) by introducing binary variables associated with each extreme point of $\Omega'$. The formulation is described as follows.
\vspace{-0.1cm}
\begin{subequations}
\label{eq:rf_11}
\allowdisplaybreaks
\begin{align}
    &\label{rf_11_obj} {\mcal{Q}}_v := \min_{\bsy{y}, \bsy{\lambda}, \eta} \bsy{a}^T\bsy{y} + \bsy{\mu}^T{\bsy{\lambda}} + \eta \\
    &\quad \textit{s.t.} \ \ \eqref{suc_0} \nonumber \\
    & \text{where} \nonumber\\
    &\label{sp_0} \quad \eta \geq \max_{\bsy{\gamma}, \bsy{\pi}, \bsy{\widetilde{\omega}}} (\bsy{h}-\bsy{Gy})^T\bsy{\gamma} + \bsy{y}^T\big(\sum^{|\bsy{\widetilde{\omega}}|}_{i=1}\sum^{|\mcal{K}|}_{j=1}(\bsy{T}^i)^T\bsy{\zeta}_{j}\widetilde{\omega}^*_{ij}\big) - \bsy{\lambda}^T\bsy{\widetilde{\omega}} \\
    &\label{sp_1}\quad\quad \textit{s.t.} \ \bsy{E}^T\bsy{\gamma}  + \bsy{W}^T\bsy{\pi} \leq \bsy{c} \\
    &\label{sp_2}\quad \qquad \quad \bsy{\gamma} \geq \bsy{0} \\
    &\label{sp_3}\quad \qquad\quad {\zeta}_{jl} \in \langle \beta_{j}, \ {\pi_l} \rangle^{MC} \qquad \forall j \in \mcal{K}, \ \forall l \in \{1,...,|(\ref{suc_2})|\} \\
    &\label{sp_4}\quad \qquad \quad \bsy{\widetilde{\omega}} = \sum^{|\mcal{K}|}_{j=1}{\beta}_{j}\bsy{\widetilde{\omega}}^{*}_j \\
    &\label{sp_5}\quad \qquad \quad \sum^{|\mcal{K}|}_{j =1}{\beta}_{j} = 1 \\
    &\label{sp_6}\quad \qquad \quad {\beta}_j \in \{0,1\}, \qquad \forall j \in \mcal{K}
\end{align}
\end{subequations}%
\noindent where 
$\mcal{K}$ is the set of vertices (denoted as $\bsy{\widetilde{\omega}}^*$) of set $\Omega'$
${\beta}_j$ are binary variables used to control the selection of an extreme point $\bsy{\widetilde{\omega}}^*_j$. $\bsy{\zeta}_j$ are continuous variables introduced for linearizing bilinear products $\beta_{j}\bsy{\pi}$ via McCormick Relaxations as described in \eqref{eq:SMC}. Note that this linearization is exact due to $\beta_j \in \{0,1\}$ \cite{Nagarajan2018lego,nagarajan2017adaptive}. The underlying idea of this MIP model is that we only consider the extreme points of $\Omega$ as candidate solutions for $\bsy{\wt{\omega}}$. Hence, for any given first-stage decision $(\bsy{y}, \bsy{\lambda})$, the resulting worst-case $\bsy{\widetilde{\omega}}$ is a vertex of $\Omega'$.}

Since $\mcal{K}$ collects the vertices of $\Omega'$ and $\mcal{K} \neq \emptyset$, lemma \ref{prop_1} proves that there exists a vertex $\bsy{\wh{\omega}}^*_j \in \mcal{K}$ which is the worst-case scenario of $\bsy{\wh{\omega}} \in \Omega'$, such that:
\begin{equation}
    \max_{\forall \bsy{\widetilde{\omega}}\in\mcal{K}}\big\{\mcal{H}(\bsy{y}, \bsy{\widetilde{\omega}}) - \bsy{\lambda}^T\bsy{\widetilde{\omega}}\big\} = \max_{\forall \bsy{\widetilde{\omega}}\in \Omega'}\big\{\mcal{H}(\bsy{y}, \bsy{\widetilde{\omega}}) - \bsy{\lambda}^T\bsy{\widetilde{\omega}}\big\}
\end{equation}
In other words, we reduce the feasible set of $\bsy{\wh{\omega}}$ to a series of extreme points in $\mcal{K}$ and solve the reformulation $\mcal{Q}_v$ in \eqref{eq:rf_11} to identify the minimum worst-case expected total cost. Recall that formulation \eqref{eq:rf_0} (denoted as $\mcal{Q}_d$) minimizes the worst-case expected total cost over ambiguity set $\Omega$. Thus, the optimal solution of problem $\mcal{Q}_v$ \eqref{eq:rf_11} is less than the minimum cost obtained by solving $\mcal{Q}_d$, as $\Omega'$ is a subset of $\Omega$. In addition, since the initial DR-OTSGMD model $\mcal{Q}_o$ is nonlinear and model $\mcal{Q}_d$ is a linear relaxation of $\mcal{Q}_o$ as described in Section \ref{OA}, the optimal solution of $\mcal{Q}_d$ is a lower bound of $\mcal{Q}_o$. Overall, $\mcal{Q}_v$ returns a lower bound of $\mcal{Q}_d$, while $\mcal{Q}_d$ generates a lower bound of $\mcal{Q}_o$. We prove this conclusion in proposition \ref{prop_2} and \ref{prop_3} as follows.
}
\begin{prop}
\label{prop_2} The worst-case expected total cost obtained from formulation $\mcal{Q}_v$ is a lower bound to the reformulated two-stage DR-OTSGMD problem $\mcal{Q}_d$.
\end{prop}

\begin{proof}
Let $\mcal{Q}^*_d$ and $\mcal{Q}^*_v$ be the optimal objective values to formulations $\mcal{Q}_d$ and $\mcal{Q}_v$, respectively. $\eta^*_d$ and $\eta^*_v$ represent the optimal values of $\eta$ to models $\mcal{Q}_d$ and $\mcal{Q}_v$, respectively. Since $\mcal{K}$ represents the vertexes of $\Omega'$, then $\mcal{K} \subseteq \Omega' \subseteq \Omega$ and the following holds: 
\begin{equation}
    \max_{\forall \bsy{\widetilde{\omega}}\in\Omega}\big\{\mcal{H}(\bsy{y}, \bsy{\widetilde{\omega}}) - \bsy{\lambda}^T\bsy{\widetilde{\omega}}\big\} \geq \max_{\forall \bsy{\widetilde{\omega}}\in\mcal{K}}\big\{\mcal{H}(\bsy{y}, \bsy{\widetilde{\omega}}) - \bsy{\lambda}^T\bsy{\widetilde{\omega}}\big\}
\end{equation}
\noindent As a result, for any given first-stage decisions $\bsy{y}$ and $\bsy{\lambda}$, the following is true:  $\bsy{a}^T\bsy{y} + \bsy{\mu}^T\bsy{\lambda} + \eta^*_d \geq \bsy{a}^T\bsy{y} + \bsy{\mu}^T\bsy{\lambda} + \eta^*_v$ (see constraints (\ref{optimality_cut}) and (\ref{sp_0})), i.e., $\mcal{Q}^*_v \leq \mcal{Q}^*_d$.  
\end{proof}
\begin{prop}Solving formulation $\mcal{Q}_v$ yields a lower bound of the worst-case expected total cost to the two-stage DR-OTSGMD problem $\mcal{Q}_o$.
\label{prop_3}
\end{prop}

\begin{proof}
Let $\mcal{Q}^*_o$, $\mcal{Q}^*_d$, and $\mcal{Q}^*_v$ be the optimal objective function values of formulations $\mcal{Q}_o$, $\mcal{Q}_d$, and $\mcal{Q}_v$, respectively. As described in Section \ref{OA}, $\mcal{Q}_d$ is a linear relaxation of  $\mcal{Q}_o$. Hence, $\mcal{Q}_d$ yields a lower bound for $\mcal{Q}_o$, i.e., $\mcal{Q}^*_d \leq \mcal{Q}^*_o$. Further, $\mcal{Q}^*_v \leq \mcal{Q}^*_d$ based on proposition \ref{prop_2}. Thus, $\mcal{Q}^*_v \leq \mcal{Q}^*_d \leq \mcal{Q}^*_o$.
\end{proof}

\section{Solution Methodology: A column-and-constraint generation algorithm}\label{sec:methodology}

To solve the two-stage DR-OTSGMD problem, we use the column-and-constraint generation (C$\&$CG) algorithm described in \cite{zeng2013solving}. Similar to Benders' decomposition, the C$\&$CG algorithm is a cutting plane method. It iteratively refines the feasible domain of a problem by sequentially generating a set of recourse variables and their associated constraints. The algorithmic description of the C$\&$CG procedure is presented in Algorithm \ref{alg:CCG}. In this algorithm, the notation $\mcal{O}$ is used to denote a subset of $\mcal{K}$. LB and UB represent incumbent lower and an upper bounds of formulation (\ref{eq:rf_11}), respectively. Here, $\epsilon$ is a sufficiently small positive constant used to determine convergence. Lines \ref{ch4_while_begin}-\ref{ch4_while_end} are the main body of the C$\&$CG and they describe a cutting plane procedure for the first $\fr{K}$ iterations. In lines \ref{ch4_solve_mp}-\ref{ch4_record_mp}, the LB is evaluated during iteration $\fr{K}$ using the incumbent solution $\widehat{\bsy{y}}^{\fr{K}+1}$, $\widehat{\bsy{\lambda}}^{\fr{K}+1}$ and $\widehat{\eta}^{\fr{K}+1}$. Note that in the initial iteration ($\fr{K}=0$), set $\mcal{O}$ is empty and master problem $\mcal{P}(\cdot)$ lacks constraints (\ref{ch4_ccg1})-(\ref{ch4_ccg3}). In lines \ref{ch4_solve_sp}-\ref{ch4_save_sp}, the subproblem $\mcal{Q}(\cdot)$ is solved, the UB is estimated, and the worst-case scenario $\widehat{\widetilde{\omega}}^{\fr{K}+1}$ is generated using the incumbent decision ($\widehat{\bsy{y}}^{\fr{K}+1}$, $\widehat{\bsy{\lambda}}^{\fr{K}+1}$). Line \ref{ch4_ccg_cuts} generates a new set of recourse variables $\bsy{x}^{\fr{K}+1}$ and associated constraints (\ref{ch4_eq:ccg_cuts}) for $\mcal{P}(\cdot)$. Note that constraint (\ref{ch4_ccg_cut_1}) is valid only if subproblem $\mcal{Q}(\cdot)$ is bounded (i.e., complete recourse). 
Line \ref{ch4_back} expands set $\mcal{O}$ and adds these newly-generated variables and constraints to the master problem $\mcal{P}(\cdot)$ to tighten the feasible domain of the first-stage variables (line \ref{ch4_solve_mp}). The process is repeated until the objective values of the upper and lower bound converge (line \ref{ch4_while_begin}). 
\begin{algorithm}[h]
\allowdisplaybreaks
\caption{Column-and-Constraint Generation (CCG)}
\label{alg:CCG}
\begin{algorithmic}[1]
    \Function{CCG}{}
      \State Set $ \fr{K} \gets 0$, $LB \gets -\infty$, $UB \gets +\infty$, $\mcal{O} = \emptyset$
    
      \While{$\frac{|UB-LB|}{LB} \leq \epsilon$} \label{ch4_while_begin}
        
      \State \label{ch4_solve_mp} Solve the following master problem
      \begin{subequations}
      \small
      \allowdisplaybreaks
        \begin{align}
            & {\mcal{P}}(\widehat{\widetilde{\bsy{\omega}}} \in \mcal{O} ) = \min_{\bsy{y}, \bsy{\lambda}, \eta, \bsy{x}} \bsy{a}^T\bsy{y} + \bsy{\mu}^T{\bsy{\lambda}} + \eta \nonumber \\
            &\quad \textit{s.t.} \  \ (\ref{suc_0}) \nonumber \\
            &\label{ch4_ccg1}\quad \qquad \eta \geq \bsy{c}^T\bsy{x}^l - {\bsy{\lambda}^T}\widehat{\widetilde{\bsy{\omega}}}^l \qquad \forall l \leq \fr{K}  \\
            &\label{ch4_ccg2}\quad \qquad \bsy{Gy} + \bsy{Ex}^l \geq \bsy{h} \qquad \forall l \leq \fr{K}  \\
            &\label{ch4_ccg3}\quad \qquad \bsy{T}(\widehat{\widetilde{\bsy{\omega}}}^{l})\bsy{y} = \bsy{Wx}^l \qquad \forall l \leq \fr{K}  
        \end{align}
      \end{subequations} 
      \State \label{ch4_record_mp} Use the optimal solution $\widehat{\bsy{y}}^{\fr{K}+1}$, $\widehat{\bsy{\lambda}}^{\fr{K}+1}$ and $\widehat{\eta}^{\fr{K}+1}$ to calculate $LB \gets \mcal{P}(\wh{\wt{\bsy{\omega}}}\in \mcal{O})$

      \State \label{ch4_solve_sp} Solve
      \begin{subequations}
      \small
      \allowdisplaybreaks
      \begin{align}
          &\qquad \quad {\mcal{Q}}(\widehat{\bsy{y}}^{\fr{K}+1}, \widehat{\bsy{\lambda}}^{\fr{K}+1}) = \max_{\bsy{\gamma}, \bsy{\pi}, \bsy{\widetilde{\omega}}} (\bsy{h}-\bsy{G}\wh{\bsy{y}}^{\fr{K}+1})^T\bsy{\gamma} + (\widehat{\bsy{y}}^{\fr{K}+1})^T\big(\sum^{|\bsy{\widetilde{\omega}}|}_{i=1}\sum^{|\mcal{K}|}_{j=1}(\bsy{T}^i)^T\bsy{\zeta}_{j}\widetilde{\omega}^*_{ij}\big) - (\widehat{\bsy{\lambda}}^{\fr{K}+1})^T\bsy{\widetilde{\omega}}  \\
          &\qquad \quad \quad \textit{s.t.} \ (\ref{sp_1})-(\ref{sp_6}) \ \nonumber
      \end{align}
      \end{subequations}
       
      \State\label{ch4_save_sp} Use the optimal solution $\widehat{\widetilde{\bsy{\omega}}}^{\fr{K}+1}$ to calculate \small $UB \gets LB - \widehat{\eta}^{\fr{K}+1} + {\mcal{Q}}(\widehat{\bsy{y}}^{\fr{K}+1}, \widehat{\bsy{\lambda}}^{\fr{K}+1})$
       
      \State \label{ch4_ccg_cuts} Generate a new set of recourse variables $\bsy{x}^{\fr{K}+1}$ and the following constraints for $\mcal{P(\cdot)}$ \begin{subequations}
      \label{ch4_eq:ccg_cuts}
      \small
      \allowdisplaybreaks
        \begin{align}
            & \label{ch4_ccg_cut_1}\eta \geq \bsy{c}^T\bsy{x}^{\fr{K}+1} - {\bsy{\lambda}^T}\widehat{\widetilde{\bsy{\omega}}}^{\fr{K}+1} \\
            & \label{ch4_ccg_cut_2}\bsy{Gy} + \bsy{Ex}^{\fr{K}+1} \geq \bsy{h}  \\
            & \label{ch4_ccg_cut_3}\bsy{T}(\widehat{\widetilde{\bsy{\omega}}}^{\fr{K}+1})\bsy{y} = \bsy{Wx}^{\fr{K}+1} \qquad
        \end{align}
      \end{subequations} 
       
      \State \label{ch4_back} Update $\mcal{O} \gets \mcal{O} \bigcup \widehat{\widetilde{\bsy{\omega}}}^{\fr{K}+1}$ and $\fr{K} \gets \fr{K}+1$

      \EndWhile \label{ch4_while_end}
      \State \Return $\widehat{\bsy{y}}^{\fr{K}+1}$
    \EndFunction
\end{algorithmic}	
\end{algorithm}

Note that the optimal solution of the relaxed reformulation (\ref{eq:rf_11}) is obtained by enumerating all the possible uncertain scenarios (vertexes) in $\mcal{K}$. 
However, even though the number extreme points is linear in the problem size,
full enumeration is computationally intractable when subset $\mcal{K}$ is large. One advantage of the C$\&$CG is that the computational efforts is often significantly reduced by using a partial enumeration of non-trivial scenarios of the random variables $\bsy{\widetilde{\omega}}$. Additionally, the C$\&$CG algorithm is known to converge in a finite number of iterations since the number of extreme points in $\mcal{K}$ is finite \cite{zhao2012robust}.

\section{Case Study}\label{Sec:case study}
This section analyzes the performance of 
our approach on 
a power system that is exposed to uncertain  geo-electric fields induced by GMDs. We use a modified version of the Epri21 system \cite{PowerGMD}. 
The size of this network
is comparable to previous work \cite{zhu2015blocking} that considered minimization of quasi-static GICs and did not consider ACOPF with topology control. Computational experiments were performed using the HPC Palmetto cluster at Clemson University with Intel Xeon E5-2665, 24 cores and 120GB of memory. All algorithms and models are implemented in Julia using using JuMP \cite{dunning2017jump}. All problems are solved using CPLEX 12.7.0 (default options).

\subsection{Data Collection and Analysis}
The main source of historical data for GMDs is recent work by Woodroffe \textit{et al.}\cite{woodroffe2016latitudinal}. The authors analyzed 100 years of data related to storms and their paper indicates that the magnitude of the corresponding induced geo-electric fields varies with the magnetic latitudes. The authors also categorized geomagnetic storms into three classes, strong, severe and extreme, according to the range of geoelectromagnetic disturbances (Dst). Table \ref{tab:gmd_parameter} summarizes parameters used for describing the uncertain electric fields induced by GMDs. In this table, $\overline{\nu}^M$ denotes the maximum amplitude of geo-electric fields induced by GMDs (see Section \ref{sec:Model of Uncertainty}).  Its value is based on a $ 95\% $ confidence interval (CI) of 100-year peak GMD magnitudes \cite{woodroffe2016latitudinal} (Table 1-3 in \cite{woodroffe2016latitudinal}). The values of ${\mu}^N$ and ${\mu}^E$ are estimated via extreme value analysis of electric fields from 15 geomagnetic storms using a Generalized Extreme Value Distribution model for both northward and eastward components (i.e., $\widetilde{\nu}^N$ and $\widetilde{\nu}^E$). We also generate the vertexes of $\Omega$ by partitioning field directions between 0\degree and 180\degree spaced by 2\degree, resulting in 90 GMD extreme points.
%
\begin{table}[h]
    \renewcommand{\arraystretch}{1.02}
    \centering
    \small
    \label{gmd_parameter}
    \captionsetup{font=small}
    \caption[Peak GMD amplitudes for geomagnetic storms with different magnetic latitudes]{Peak GMD amplitudes for geomagnetic storms with different magnetic latitudes based on 100 years of historical data. $\mu^E$ and $\mu^N$ are the mean values of geo-electric fields in the eastward and northward directions. MLAT denotes the geomagnetic latitude. Dst denotes geoelectromagnetic disturbances. }
    \begin{tabular}{lccc}
        \toprule
         &\multicolumn{3}{c}{$\overline{\nu}^M$  ($\mu^N$, $\mu^E$) (V/km)}  \\
        \cmidrule(lr){2-4} 
         & Strong & Severe & Extreme \\
         MLAT &$(-100 nT \geq Dst > -200 nT)$ & $(-200 nT \geq Dst > -300 nT)$ & $(-300 nT \geq Dst)$ \\
         \midrule 
         $40\degree$--$45\degree$ & $1.6 \  (0.9, 0.8)$ & $2.0 \ (0.9, 0.8) $ & $3.5 \ (1.1, 0.9) $\\
         $45\degree$--$50\degree$ & $1.2 \ (0.7, 0.7)$ & $1.6 \ (0.8, 0.7) $ & $3.5 \ (1.5, 1.3) $\\
         $50\degree$--$55\degree$ & $3.5 \ (2.1, 1.8)$ & $5.0 \ (2.5, 2.1) $& $6.0 \ (3.1, 2.7) $ \\
         $55\degree$--$60\degree$ &
         $11.5 \ (6.6, 5.6)$ & $6.6 \ (3.7, 3.1) $& $9.1 \ (4.2, 3.6) $\\
         $60\degree$--$65\degree$ &
         $6.6 \ (5.0, 4.3)$ & $6.6 \ (4.3, 3.6) $ & $12.7 \ (5.9, 5.1) $\\
         $65\degree$--$70\degree$ &
         $8.8 \ (6.1, 5.2)$ & $8.8 \ (5.3, 4.5) $& $10.6 \ (5.8, 4.9) $ \\
         $70\degree$--$75\degree$ &
         $7.7 \ (5.1, 4.3)$ & $6.3 \ (3.9, 3.3) $& $16.1 \ (6.8, 5.8) $ \\
         \bottomrule
    \end{tabular}
    \label{tab:gmd_parameter}
\end{table}%
%

While most of the parameters for modeling GIC are present in the Epri21 test case, there was some missing data.  We next discuss the data we added or modified in the Epri21 system (Tables \ref{tb:para_1}--\ref{tab:gen_data}). First, we geolocated the system in 
Quebec, Canada. To convert Quebec's geographic latitudes into magnetic latitudes we used the method described in Appendix A, which yields a magnetic latitude of $55\degree-60\degree$.
In this model, the cost of load shedding and over-consuming load is ten times the most expensive generation cost; and the fuel cost coefficients of additional real-power generation (i.e., $\Delta_i$) is 20\% more than the reserved generation prior to a storm.

\begin{table}[h]
  \allowdisplaybreaks
  \captionsetup{font=small}
  \caption[Transformer data and transmission line data]{(a) Transformer winding resistance and $k$ are estimated based on the test cases provided in \cite{GIC2013flow, horton2012test}. (b) The nominal line length {\color{black}for the Epri21 system \cite{PowerGMD}} {\color{black}is used to derive} an approximate geospatial layout of the power system nodes.}
  \centering
  \footnotesize
  \subtable[Transformer data]{
  \setlength{\tabcolsep}{0.3em}
  \begin{tabular}{lccccccc}
  \toprule
  &&Resistance&&Resistance&&&\\
  Name&Type&W1(Ohm)&Bus No.&W2(Ohm)&Bus No.&Line No.&k (p.u.)\\
  \midrule
T 1  & Wye-Delta  & 0.1 & 1 & 0.001 & 2 & 16 & 1.2\\
T 2  & Gwye-Gwye & 0.2 & 4 & 0.1 & 3 & 17 & 1.6\\
T 3  & Gwye-Gwye & 0.2 & 4 & 0.1 & 3 & 18 & 1.6\\
T 4  & Auto & 0.04 & 3 & 0.06 & 4 & 19 & 1.6\\
T 5  & Auto & 0.04 & 3 & 0.06 & 4 & 20 & 1.6\\
T 6  & Gwye-Gwye & 0.04 & 5 & 0.06 & 20 & 21 & 1.6\\
T 7  & Gwye-Gwye & 0.04 & 5 & 0.06 & 20 & 22  & 1.6\\
T 8  & GSU & 0.15 & 6 & 0.001 & 7 & 23 & 0.8\\
T 9  & GSU & 0.15 & 6 & 0.001 & 8 & 24 & 0.8\\
T 10 & GSU & 0.1 & 12 & 0.001 & 13 & 25 & 0.8\\
T 11 & GSU & 0.1 & 12 & 0.001 & 14 & 26 & 0.8\\
T 12 & Auto & 0.04 & 16 & 0.06 & 15 & 27 & 1.1\\
T 13 & Auto & 0.04 & 16 & 0.06 & 15 & 28 & 1.1\\
T 14 & GSU & 0.1 & 17 & 0.001 & 18 & 29 &1.2\\
T 15 & GSU & 0.1 & 17 & 0.001 & 19 & 30 & 1.2\\
  \bottomrule
  \end{tabular}
  }
  \subtable[Transmission line data]{
  \setlength{\tabcolsep}{0.2em}
  \begin{tabular}{lcccc}
  \toprule
     & From & To & Length \\
    Line & Bus & Bus & (km) \\
    \midrule
    1 & 2 & 3 & 122.8 \\
    2 & 4 & 5 & 162.1\\
    3 & 4 & 5 & 162.1\\
    4 & 4 & 6 & 327.5 \\
    5 & 5 & 6 & 210.7 \\
    6 & 6 & 11 & 97.4 \\
    7 & 11 & 12 & 159.8\\
    8 & 15 & 4 & 130.0 \\
    9 & 15 & 6 & 213.5 \\
    10 & 15 & 6 & 213.5 \\
    11 & 16 & 20 & 139.2 \\
    12 & 16 & 17 & 163.2 \\
    13 & 17 & 20 & 245.8 \\
    14 & 17 & 2 & 114.5 \\
    15 & 21 & 11 & 256.4 \\
  \bottomrule
  \end{tabular}
  }
  \label{tb:para_1}
\end{table}%

\begin{table}[h]
\color{black}
  \captionsetup{font=small}
  \caption[Substation data and other parameters]{(a) The substation grounding resistance $(GR)$ is estimated from typical values of grounding resistance of substations provided in \cite{morstad2012grounding}. (b) The original line parameters $r_{e}^o$ and $x_{e}^o$ are scaled by $\beta_{e}$, a ratio from new to original line lengths.}
  \centering
  \footnotesize
  \subtable[Substation data]{
  \renewcommand{\arraystretch}{1.26}
  \setlength{\tabcolsep}{0.7em}
  \begin{tabular}{lccccccc}
  \toprule
    Name&Latitude&Longitude&GR(Ohm)\\
  \midrule
    SUB 1&46.61&-77.87&0.20 \\
    SUB 2&47.31&-76.77&0.20 \\
    SUB 3&46.96&-74.68&0.20 \\
    SUB 4&46.55&-76.27&1.00 \\
    SUB 5&45.71&-74.56&0.10 \\
    SUB 6&46.38&-72.02&0.10 \\
    SUB 7&47.25&-72.09&0.22 \\
    SUB 8&47.20&-69.98&0.10 \\
  \bottomrule
  \end{tabular}
  }
  \hspace{0.1cm}
  \subtable[Other parameters]{
  \renewcommand{\arraystretch}{1.18}
  \setlength{\tabcolsep}{0.7em}
  \begin{tabular}{l|cccc}
  \toprule
  $\kappa_i^+$ & \$ 1000 /MW (or MVar)\\
  $\kappa_i^-$ & \$ 1000 /MW (or MVar)\\
  $\underline{v}_i$ & 0.9 p.u. \\
  $\overline{v}_i$ & 1.1 p.u. \\
  $c_i^{R1}$ & $120\%c_i^1$ \\
  $c_i^{R2}$ & $120\%c_i^{2}$ \\
  $\overline{I}^a_{e}$ &  $s_{e}/\min\{\underline{v}_i, \underline{v}_j\}$\\
  $\overline{\theta}$&30\degree \\
  \bottomrule 
  \end{tabular}
  }
  \label{tb:para_2}
\end{table}%
%
%
\begin{table}[h]
    \captionsetup{font=small}
    \caption{Generator Data}
    \centering
    \small
    \begin{tabular}{lcccccccc}
    \toprule
        Name & Bus No. & $\underline{gp}$ (MW) & $\overline{gp}$ (MW) & $\underline{gq}$ (MVar) & $\overline{gq}$ (MVar) & $c^2, c^1, c^0$ (\$)\\
        \midrule
        G 1 & 1 & 472.3 & 782.3 & 51.57 & 61.57 & 0.11, 5, 60 \\
        G 2 & 7 & 595.0 & 905.0 & -56.56 & 46.56 & 0.11, 5, 60 \\
        G 3 & 8 & 595.0 & 905.0 & -56.56 & 46.56 & 0.11, 5, 60 \\
        G 4 & 13 & 195.0 & 505.0 & -10.61 & -0.61 & 0.11, 5, 60 \\
        G 5 & 14 & 195.0 & 505.0 & -10.61 & -0.61 & 0.11, 5, 60 \\
        G 6 & 18 & 295.0 & 605.0 & 18.78 & 28.78 & 0.11, 5, 60 \\
        G 7 & 19 & 295.0 & 605.0 & 18.78 & 28.78 & 0.11, 5, 60 \\
        \bottomrule
    \end{tabular}
    \label{tab:gen_data}
\end{table}

\subsection{Experimental Results}
\label{ER}

To evaluate the performance of the Epri21 system under the influence of uncertain geomagnetic storms, 
we use the three different storm levels described in Table \ref{tab:gmd_parameter}: extreme, severe and strong. We also consider results where generator ramping limits range from $0\%$ to $20\%$ of its maximum real power generation ($\overline{gp}$) in $5\%$ steps. 
We use ramping limits to loosely model warning time, i.e., how much time is available for generators to respond once the full characteristics of the storm are known.
To analyze the benefits of capturing uncertain events via the DR optimization, we study the following four cases:

%

\begin{enumerate}
\small
  \item C0: The ACOTS with GIC effects induced by the mean geo-electric fields (i.e., $\mu^E$ and $\mu^N$):\\
  {$ M_{mean} :=$ Min $\Big\{\bsy{a}^T\bsy{y} + \bsy{c}^T\bsy{x}$: \eqref{suc_0}, (\ref{suc_1})-(\ref{suc_2}); $\bsy{\widetilde{\omega}} = [\mu^E, \mu^N]^T \Big\} \Leftrightarrow \bsy{z}^*_{mean}, \bsy{\rho}^*_{mean}, \mathfrak{C}^*_{mean}$} 
  \label{case0}
  
  \item C1: The two-stage DR-OTSGMD:\\
  $M_{dr} := $ Min\Big\{(\ref{rf_11_obj}): \eqref{suc_0}, \eqref{sp_0}-\eqref{sp_6}\Big\} $\Leftrightarrow \bsy{z}^*_{dr}, \bsy{\rho}^*_{dr}, \mathfrak{C}^*_{dr}$
  \label{case1}
  
  \item C2: The two-stage DR-OTSGMD with fixed topology configuration and generator setpoints:\\
  $M_{fm} := $ Min\Big\{(\ref{rf_11_obj}): \eqref{suc_0}, \eqref{sp_0}-\eqref{sp_6}\Big\}; $\bsy{z} = \bsy{z}^*_{mean}$, $\bsy{\rho} = \bsy{\rho}^*_{mean}$\Big\} $\Leftrightarrow \bsy{z}^*_{mean}, \bsy{\rho}^*_{mean}, \mathfrak{C}^*_{fm}$
  \label{case2}
  
  \item C3: The two-stage DR-OTSGMD with fixed topology configuration:\\
  $M_{ft} := $ Min\Big\{(\ref{rf_11_obj}): \eqref{suc_0}, \eqref{sp_0}-\eqref{sp_6}\Big\}; $\bsy{z} = \bsy{1} \} $\Big\} $\Leftrightarrow \bsy{1}, \bsy{\rho}^*_{ft}, \mathfrak{C}^*_{ft}$
  \label{case3}
\end{enumerate}%

\noindent where $M_{\alpha}$ is the optimization model for case $\alpha$; $\bsy{z}_{\alpha}^*$ and $\bsy{\rho}_{\alpha}^*$ represent the optimal first-stage decisions, respectively, and $\mathfrak{C}_{\alpha}^*$ is the objective function value. The values of $\bsy{z}_{\alpha}^*$, $\bsy{\rho}^*_{\alpha}$ and $\mathfrak{C}_{\alpha}^*$ are obtained by solving $M_{\alpha}$. Case C0 identifies the optimal first-stage decisions (i.e., generation and topology) and evaluates the objective according to the mean geo-electric fields (V/km) in the northward and eastward directions. Case C2 identifies a solution for the two-stage DR-OTSGMD model assuming uncertain GMDs. Case C3 finds a solution to the worst-case expected cost $\mathfrak{C}^*_{fm}$ of the DR optimization model using the generation and topology decisions of Case C0. Finally, Case C3 is similar to Case C2, but topology reconfiguration is not allowed.

\subsubsection{Case C0: GIC mitigation for the mean geo-electric fields}

Case C0 assumes that power system operators use the mean value of a storm level (i.e., strong, severe and extreme) to optimize power generation and system topology.  In Table \ref{tb:C0}, we summarize the computational results for Case C0 under different storm levels and ramping limits. The results suggest that the cost of mitigating the impacts of a storm decreases as ramping limits (warning times) increase. The highest cost occurs when generator ramping in not allowed (i.e., $\overline{u}^R_i = 0$). This is the reason why increasing ramping limits from 0 to 20\% leads to a 42.1\% cost decrease for strong GMD events.
Additionally, the total cost varies with storm levels; increasing the intensity of geoelectric field induces a higher cost. Costs increase with storm level; however, this increase is not significant, mainly because the cost changes due to generator dispatch are small. For example, the cost difference between strong and severe storms is only 0.02\% when the ramping limit is 10\%.
\begin{table}[h]
    \centering
    \allowdisplaybreaks
    \renewcommand{\arraystretch}{1.06}
    \footnotesize
    \small
    \captionsetup{font=small}
    \caption[Computational results for Case C0]{Computational results for Case C0 with respect to different storm levels (SL) and ramping limits ($\overline{u}^R_i$). TC denotes the minimum total cost for generator dispatch and load shedding. ${\mu}^E$ and ${\mu}^N$ are the mean values of geo-electric fields for northward and eastward components, respectively. $\bsy{z}^*$ represents the optimal line switching decisions.}
    \setlength{\tabcolsep}{0.8em}
    \begin{tabular}{cccccc}
        \toprule
         SL(55$\degree$--60$\degree$)&$\overline{u}_i^R$ (\%) & TC (\$) & ${\mu}^{E}$, ${\mu}^{N}$ (V/km) & $\bsy{z}^*$ & Wall Time (sec) \\
         \midrule
         \multirow{5}{*}{\textsc{Strong}}&0 & 386,143 & 5.6, 6.6 &2, 17, 19-22, 28 & 28 \\
         &5 & 351,124 &5.6, 6.6 & 2, 18-22, 28 & 21  \\
         &10 & 320,425 &5.6, 6.6 & 2, 17, 19-22 & 24\\
         &15 & 294,077 &5.6, 6.6 & 2, 17, 19-22, 27 & 29 \\
         &20 & 271,751 &5.6, 6.6 & 2, 17, 19-22, 28 & 27\\
         \cmidrule(lr){1-6}
         \multirow{5}{*}{\textsc{Severe}}&0 & 386,091 & 3.1, 3.7 & 2, 13, 15, 18, 21, 22 & 16  \\
         &5 & 351,066 & 3.1, 3.7 & 15, 18, 20-22 & 6 \\
         &10 & 320,363 & 3.1, 3.7 & 2, 15, 18, 21, 22 & 19 \\
         &15 & 294,022 & 3.1, 3.7 & 2, 15, 19-22 & 19    \\
         &20 & 271,696 & 3.1, 3.7 &15, 19-22, 28 & 22 \\
         \cmidrule(lr){1-6}
          \multirow{5}{*}{\textsc{Extreme}}&0 & 386,099 & 3.6, 4.2 &2, 15, 19-22, 27 & 17 \\
         &5 & 351,079 & 3.6, 4.2 & 15, 17, 19-22, 27 & 23 \\
         &10 & 320,379 & 3.6, 4.2 & 2, 15, 17, 19-22 & 19 \\
         &15 & 294,039  & 3.6, 4.2 &15, 17-19, 21, 22, 27 & 6 \\
         &20 & 271,709  & 3.6, 4.2 &2, 15, 19-22, 28 & 17  \\
         \bottomrule
    \end{tabular}
    \label{tb:C0}
\end{table}

\subsubsection{Case C1: DR optimization for GIC mitigation under uncertainty }

In Case C1, uncertain GMD events are modeled via an ambiguity set and we relax the curve of maximum storm strength to 90 extreme points. 
Similar to Case C0, we conduct sensitivity analysis with respect to storm levels and ramping limits. Table (\ref{tb:C1}) summarizes these computational results.
\vspace{0.1cm}
\begin{table}[h]
    \centering
    \allowdisplaybreaks
    \renewcommand{\arraystretch}{1.01}
    \footnotesize
    \small
    \captionsetup{font=small}
    \caption[Computational results for Case C1]{Computational results for Case C1 with respect to different strom levels (SL) and ramping limits ($\overline{u}^R_i$). WETC denotes the worst-case expected total cost. $\widetilde{\nu}^{*E}$, $\widetilde{\nu}^{*N}$ represent the worst-case geo-electric fields for eastward and northward components, respectively. $\bsy{z}^*$ represents the optimal line switching decisions.}
    \setlength{\tabcolsep}{0.5em}
    \renewcommand{\arraystretch}{1.1}
    \begin{tabular}{cccccc}
        \toprule
        SL(55$\degree$--60$\degree$)&$\overline{u}_i^R$ (\%) & WETC (\$) & $\widetilde{\nu}^{*E}$, $\widetilde{\nu}^{*N}$ (V/km) & $\bsy{z}^*$ & Wall Time (sec) \\
         \midrule
         \multirow{5}{*}{\textsc{Strong}}&0 & 533,625 & 4.3, 10.7 &8, 9, 17-22, 27& 3,995\\
         &5 & 486,038 & 4.3, 10.7 & 8, 9, 17-22, 28 & 8,738  \\
         &10 & 455,694 & 4.3, 10.7 & 8, 9, 17-22, 28 & 7,596 \\
         &15 & 428,500 & 4.3, 10.7 & 8, 10, 17-22, 28 & 9,693 \\
         &20 & 406,176 & 4.3, 10.7 & 8, 9, 17-22, 28 & 8,049 \\
        \cmidrule(lr){1-6}
         \multirow{5}{*}{\textsc{Severe}}&0 & 386,342 & 4.7, 4.6 & 2, 8, 18-20, 22, 27, 28 & 1,242 \\
         & 5 & 351,323 & 4.7, 4.6 & 2, 8, 18-21, 27, 28 & 1,061\\
         &10 & 320,622 & 4.7, 4.6 & 2, 8, 17, 19, 20, 22, 27, 28 & 1,364\\
         &15 & 294,249 & 4.7, 4.6 & 2, 8, 17, 19, 20, 22, 27, 28 & 1,347 \\
         &20 & 271,923 & 4.7, 4.6 &2, 8, 18-21, 27, 28 &1,101 \\
         \cmidrule(lr){1-6}
         \multirow{5}{*}{\textsc{Extreme}}&0 & 467,117 &8.4, 3.4 & 2, 3, 17-20, 22, 27, 28 & 4,896 \\
         &5 & 429,554 & 8.4, 3.4 & 2, 3, 17-21, 27, 28 & 3,146   \\
         &10 & 399,898 &8.4, 3.4 & 2, 3, 17-21, 27, 28 & 3,138\\
         &15 & 373,116 &8.4, 3.4 &2, 3, 17-20, 22, 27, 28& 4,838 \\
         &20 & 351,579 & 8.4, 3.4 & 2, 3, 17-21, 27, 28 & 4,185 \\
         \bottomrule
    \end{tabular}
    \label{tb:C1}
\end{table}%
The results show that, similar to Case C0, the worst-case expected total cost (WETC) decreases as the ramping limits (warning times) increase. This decrease is as large as 31.4\% for a strong storm (i.e., $\overline{u}_i^R=20\%$ versus $\overline{u}_i^R=0\%$).

Moreover, we observed that, for a given storm level, the worst-case geo-electric field remains the same at different ramping limits. For example, in extreme storms, $\widetilde{\nu}^{*E}$ and $\widetilde{\nu}^{*N}$ is always 8.4  and 3.4 V/km, respectively. This indicates that, for a given storm level, the worst scenario (empirically) for this system is the same extreme point of $\Omega$. In addition, the difference in the WETC between strong and extreme (and severe) events is considerable. 
For example, the cost difference between strong and severe events is 27.6\% when $\overline{u}^R_i=0$. 
Table \ref{tb:LS} reports the amount and percentage of the WETC due to load shedding. The results suggest that, for Case C1, the increase in WETC under a strong event is primarily due to changes in load shedding. 
An average 13.11\% of the WETC is due to load shedding.

\subsubsection{Case Comparisons: Cost Benefits of the DR optimization}

Figure \ref{Fig:all_cases} summarizes the total cost $\mathfrak{C}^*$ for all cases defined in Section (\ref{ER}). The DR optimization model results in higher costs for all problems solved. These costs are highest when the storm level is strong.
\begin{figure}[h]
  \centering
  \captionsetup{font=small}
  \subfigure[Strong]{
  \label{Fig:c0}
  \includegraphics[scale=0.3]{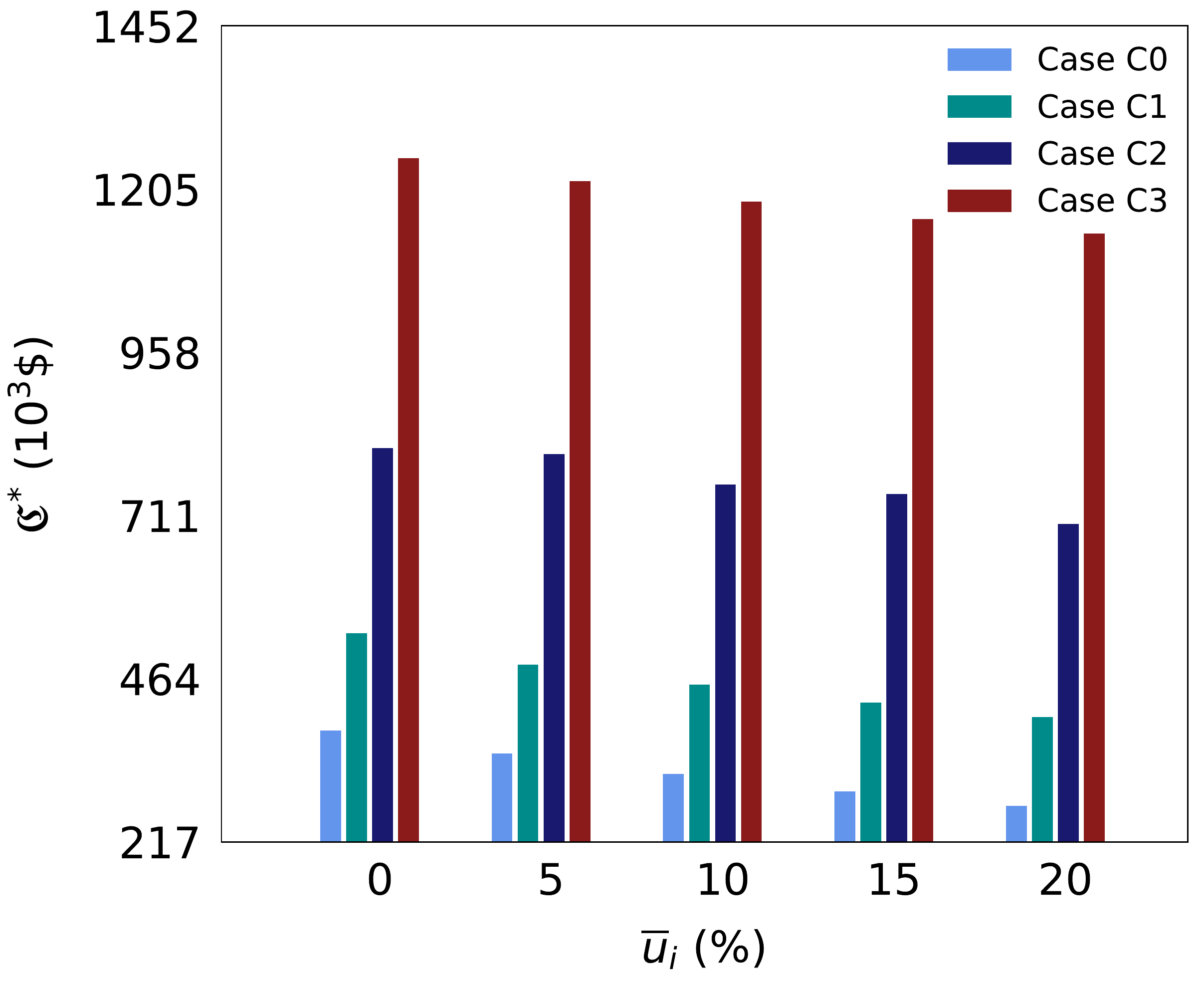}}       
  \subfigure[Severe]{\label{fig:severe_all}
  \includegraphics[scale=0.3]{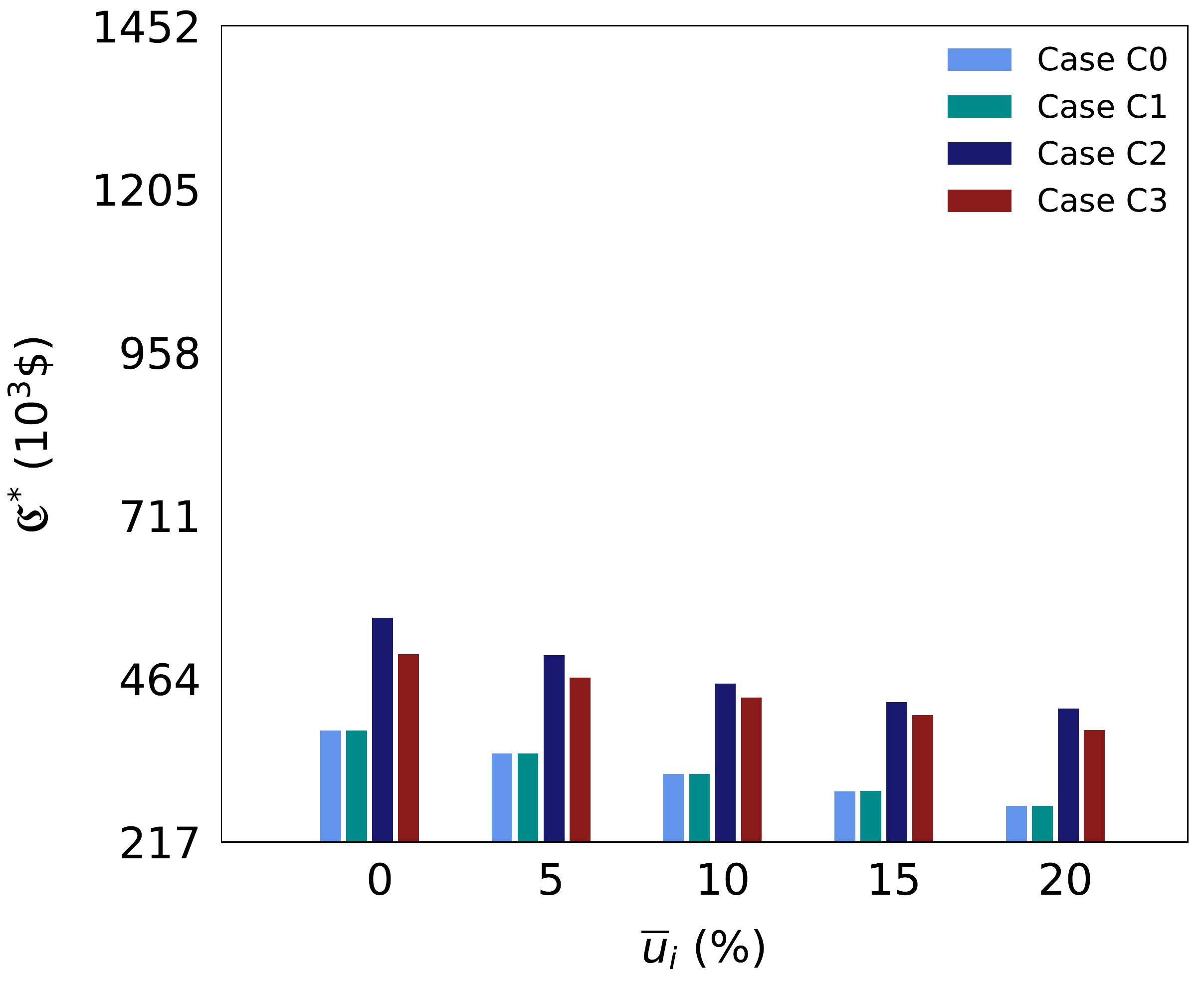}} 
    \subfigure[Extreme]{\label{fig:extreme_all}
  \includegraphics[scale=0.3]{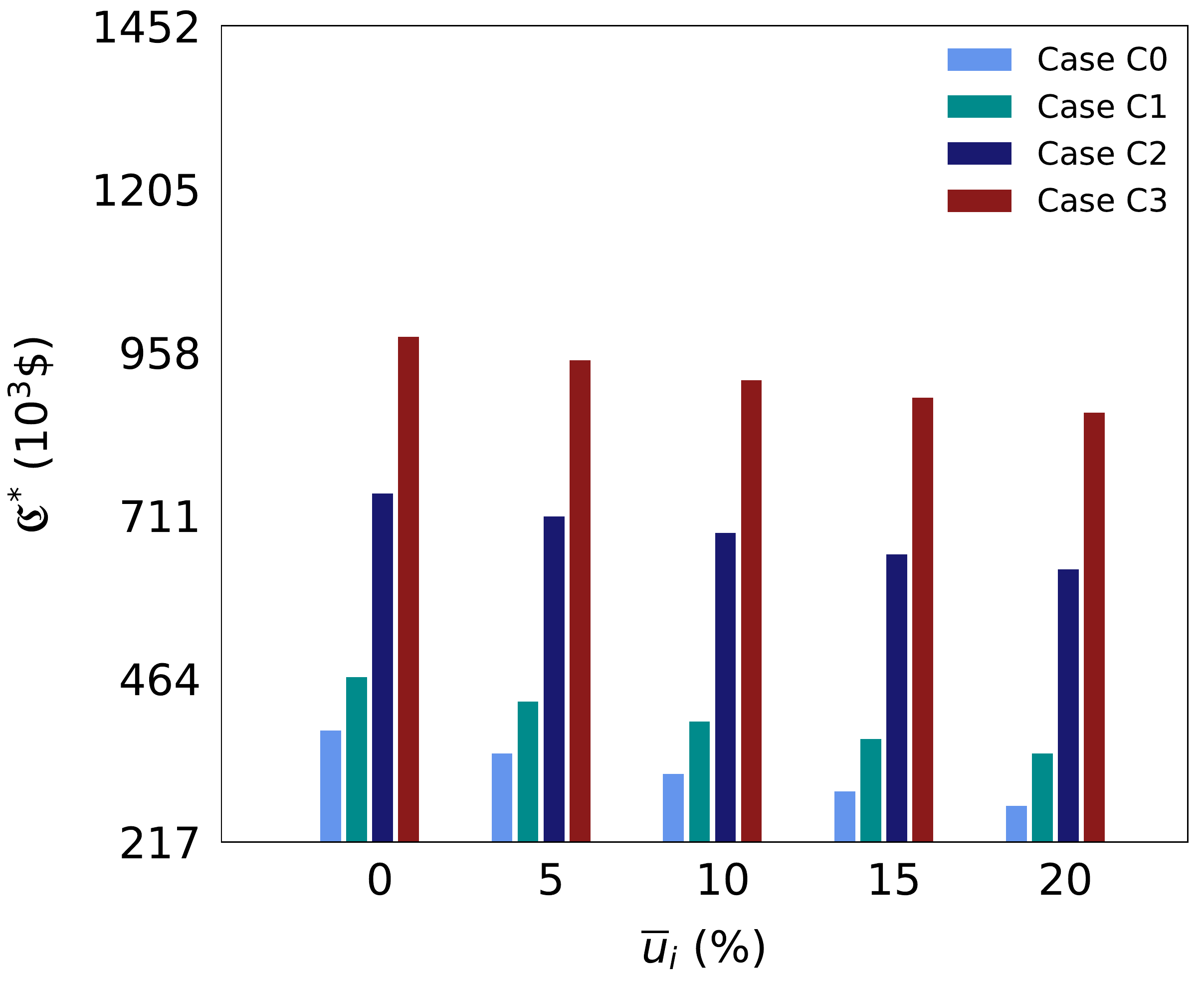}} 
  \caption[Cost comparisons among all cases.]{Cost comparisons among all cases.} 
  \label{Fig:all_cases}
\end{figure}
Tables \ref{tb:C0} and \ref{tb:C1} indicate that the computation time required for solving Case C1 is higher than Case C0. Hence, modeling uncertainty in geo-electric fields results in a significant increases in computational efforts, but it also provides significant 
robustness when compared to deterministic models which ignore uncertainty. 


We next compare the cost of making first stage decisions based on the mean GIC distributionally robust (C2) with first stage decision based on the full distributionally robust model (C1). 
Similar to Case C1, the worst-case expected cost of Case C2 decreases as ramping limits increase. For all storm levels, the cost benefits of the DR-OTSGMD vary with ramping limits and are statistically significant. For example, during a strong storm the cost savings are as much as 68.6\% (i.e., $(\mathfrak{C}^*_{fm}-\mathfrak{C}^*_{dr})/\mathfrak{C}^*_{fm}$). This is because the generation and topology decisions in Case C1 are determined by the DR model; however, these decisions for Case C2 are fixed to Case C0.

We also compare Case C1 with C3 to evaluate the benefits of network reconfiguration. The results displayed in Figure \ref{Fig:all_cases} suggest that line switching decisions significantly lower the cost of GIC mitigation under uncertainty, in particular for strong and extreme storms. For example, Figure \ref{Fig:c0} shows that the WETC is higher than Case C2 and has significantly increased when line switching is not allowed in the two-stage DR-OTSGMD model. Notice that when $\overline{u}^R_i = 10\%$, the cost increase in Case C3, as compared to C1, is 88.4\% and the increase in cost in Case C2, as compared to C1, is 42.8\%.
%

Table \ref{tb:LS} demonstrates the impacts of load shedding on cost. The results indicate that the cost differences between any two cases are primarily because of load shedding. For example, in the case of extreme storms, the total cost of Case C2 is much higher than Case C1. This is due to the observation that the topology control in Case C1 leads to a load shedding cost of 6.99\% on average. In contrast, the fixed topology in Case C2 and C3 results in a load shedding cost of 23.51\% and 61.78\% (on average), respectively. Similarly, for severe storms, load shedding is observed only in Case C2 and C3. As a consequence, Case C2 yields the highest cost in comparison to the other three cases.
\begin{table}[!h]
  \centering
  \allowdisplaybreaks
  \renewcommand{\arraystretch}{1.0}
  \footnotesize
  \small
  \captionsetup{font=small}
  \caption[Load shedding cost for all cases]{Load shedding costs for all cases. The average (Avg.), minimum (Min.), maximum (Max.) and the standard deviation (Std.) of load shedding cost (LSC) are computed over ramping limits from 0\% -- 20\%. For Case C1, C2 and C3, LSC is associated with the worst-case geo-electric fields.
  Solutions displayed in parentheses denote the percentage of the total cost ( $\mathfrak{C}^*$) due to load shedding. ``-" indicates that no load shedding is observed.} \label{tab:loadshedding} 
  \setlength{\tabcolsep}{0.8em}
  \renewcommand{\arraystretch}{1.05}
  \begin{tabular}{lrrrrrrrrr}
  \toprule
      &&\multicolumn{4}{c}{ LSC(\% of the $\mathfrak{C}^*$)}\\
      \cmidrule(lr){3-6}
      SL(55$\degree$--60$\degree$)&Case& Avg. & Min. & Max. &Std.  \\ 
      \midrule
      \multirow{4}{*}{\textsc{Strong}}
      &C0 & - & - & - & - \\
      &C1 & 72,566 (13.11) & 52,190 (12.23) & 105,943 (19.85) & 20,600 (2.87) \\
      &C2 & 160,174 (23.51)& 148,672 (20.95) & 166,419 (26.40) & 6,135 (2.07)   \\
      &C3 & 735,028 (61.78) & 734,609 (58.70) & 735,337 (64.57) & 263 (2.07) \\
      \midrule
      \multirow{4}{*}{\textsc{Severe}}
      &C0 & - & - & - & - \\
      &C1 & - & - & - & -  \\
      &C2 & 701 (0.15) & 520 (0.11) & 983 (0.20) & 179 (0.03)  \\
      &C3 & 67 (0.01) & - & 333 (0.07) & 133 (0.03)\\
      \midrule
      \multirow{4}{*}{\textsc{Extreme}}
      &C0 & - & - & - & - \\
      &C1 & 30,058 (6.99) & - & 75,932 (17.68) & 36,814 (8.57) \\
      &C2 & 140,909 (18.45) & 122,134 (17.20) & 155,781 (20.53) & 11,367 (1.17) \\
      &C3 & 431,128 (45.07) & 374,709 (40.90) & 485,665 (49.46) & 45,865 (3.65) \\
      \bottomrule
 \end{tabular}
 \label{tb:LS}
\end{table}%

\subsubsection{Performance of the CCG Algorithm}

One advantage of the CCG algorithm is the reduced computational cost associated with partial enumeration of significant extreme points. Table \ref{tb:all} summarizes the computational performance of enumerating all vertexes of $\Omega$ to evaluate the worst-case expected total cost. In other words, all second-stage variables and constraints are added to the master problem simultaneously, $\mcal{P^\mcal{K}}$. 
\vspace{-0.7cm}
\begin{subequations}
\allowdisplaybreaks
\begin{align}
    & {\mcal{P^\mcal{K}}} = \min_{\bsy{y}, \bsy{\lambda}, \eta, \bsy{x}} \bsy{a}^T\bsy{y} + \bsy{\mu}^T{\bsy{\lambda}} + \eta \nonumber \\
    &\quad \textit{s.t.} \  \ (\ref{suc_0}) \nonumber \\
    &\label{ch4_ccg1_all}\quad \qquad \eta \geq \bsy{c}^T\bsy{x}^l - {\bsy{\lambda}^T}{\widetilde{\bsy{\omega}}}^l \qquad \forall l \in \mathcal{K}  \\
    &\label{ch4_ccg2_all}\quad \qquad \bsy{Gy} + \bsy{Ex}^l \geq \bsy{h} \qquad \forall l \in \mathcal{K}  \\
    &\label{ch4_ccg3_all}\quad \qquad \bsy{T}({\widetilde{\bsy{\omega}}}^{l})\bsy{y} = \bsy{Wx}^l \qquad \forall l \in \mathcal{K}
\end{align}
\end{subequations}

In Table \ref{tb:all}, we focus on severe storms and solve this problem with different size of extreme points: 3, 9 and 90 (which are generated by partitioning field directions between 0$\degree$ and 180$\degree$ spaced by 60$\degree$, 20$\degree$ and 2$\degree$, respectively). We also set the ``time-out" parameter to 3000 seconds which is about two times the maximum computational time for severe storms in Table \ref{tb:C1}. The results show that when $|\mcal{K}|$ is 90, no solution is found for all values of $\overline{\mu}_i^R$ when the time limit is reached. In contrast, optimal solutions can be found by the CCG within 1400 seconds. In addition, the problem can not even be solved to optimality when there are only 9 extreme points. Even worse, when $\overline{u}_i^R$ equals 10\%, no solution can be found within the time limit. Note that the objective shown in Table \ref{tb:all} is the best known lower bound for $\mcal{P}^\mcal{K}$, thus a larger value indicate a better performance. We also observe that all optimal solutions can be found when the size of extreme points is reduced to 3. However, the solution quality is significantly worse (i.e., the objective value is lower than the results in Table \ref{tb:C1}) since the relaxation of the uncertainty set is too loose.

\begin{table}[h]
  \centering
  \allowdisplaybreaks
  \renewcommand{\arraystretch}{1.15}
  \footnotesize
  \small
  \captionsetup{font=small}
  \caption[Computational results obtained by enumerating all extreme points of $\Omega$.]{Computational results obtained by enumerating all extreme points of $\Omega$. T represents the computation time in seconds and "TO" indicates time-out. The best objective value (bound) found within the time limit is shown under column $\mathfrak{C}^*$ and its relative MIP gap ($\%$) is in parentheses. }
  \setlength{\tabcolsep}{0.8em}
  \begin{tabular}{ccrrrrrrrrrr}
  \toprule
      &&\multicolumn{2}{c}{3 extreme points} &\multicolumn{2}{c}{9 extreme points}&\multicolumn{2}{c}{90 extreme points}\\
      \cmidrule(lr){3-4}
      \cmidrule(lr){5-6}
      \cmidrule(lr){7-8}
      SL(55$\degree$--60$\degree$)&$\overline{u}_i^R$ (\%)& $\mathfrak{C}^*$ (gap$\%$) & $T$ & $\mathfrak{C}^*$ (gap$\%$) &$T$ & $\mathfrak{C}^*$ (gap$\%$) & $T$  \\ 
      \midrule
      \multirow{7}{*}{\textsc{Severe}}
      &0 & 386,255 (0.0) & 1365 & 386,067 (11.1) & TO  & NaN (--) & TO \\
      &5 & 351,236 (0.0) & 1003 & 351,048 (12.6) & TO & NaN (--) & TO  \\
      &10 & 320,536 (0.0) & 592  & NaN (-) & TO & NaN (--) & TO   \\
      &15 & 294,174 (0.0) & 634  & 294,010 (8.9) & TO & NaN (--) & TO\\
      &20 & 271,849 (0.0) & 617 & 271,685 (39.3) & TO & NaN (--) & TO\\
      \bottomrule
 \end{tabular}
 \label{tb:all}
\end{table}%

\section{Conclusions and Future Research}
\label{sec:conclusions}

In this paper, we developed a two-stage DR-OTSGMD formulation for solving OTS problems that include reactive power consumption induced by uncertain geo-electric fields. Given a lack of probability distributions to model geo-magnetic storms, we construct an ambiguity set to characterize a set of probability distributions of the geo-electric fields, and to minimize the worst-case expected total cost over all geo-electric field distributions defined in the ambiguity set. Since the DR formulation is hard to solve, we derive a relaxed reformulation that yields a decomposition framework for solving our problem based on the CCG algorithm. We prove that solving this reformulation yields a lower bound of the original DR model. The case studies based on the modified Epri21 system show that modeling the uncertainty in the GMD-induced geo-electric field is crucial and the DR optimization method is an effective approach for handling this uncertainty. 

There remain a number of interesting future directions. For example, considering additional moment information, such as the variations of random variables, could enhance the modeling of the ambiguity set. Additionally, new approaches are needed to scale the algorithm to larger scale instances. Another potential extension extends the formulation to integrate N-1 security (contingency) constraints in order to increase the resiliency of transmission systems \cite{nagarajan2016optimal} during GMD extreme events. 

\section*{Acknowledgements}

The  work  was  funded  by  the Los Alamos National Laboratory LDRD project
\emph{Impacts of Extreme Space Weather Events on Power Grid Infrastructure: Physics-Based Modelling of Geomagnetically-Induced Currents (GICs) During Carrington-Class Geomagnetic Storms}.
 Los Alamos National Laboratory is operated by Triad National Security, LLC, for the National Nuclear Security Administration of U.S. Department of Energy (Contract No. 89233218CNA000001).
under
the auspices of the NNSA of the U.S. DOE at LANL under Contract No.  DE-
AC52-06NA25396.




\clearpage

\begin{appendices}
\section{Conversion from Magnetic to Geographic Coordinates}
\begin{table}[!htp]
    \centering
    \footnotesize
    \caption{The dipole coefficients ($g_1^0$, $g_1^1$, $h_1^1$) from the International Geophysical Reference Field (IGRF).}
    \begin{tabular}{lccc}
    \toprule
    Epoch& $g_1^0$ & $g_1^1$ & $h_1^1$ \\
    \midrule
1965    &   -30334  & -2119  &   5776\\
1970    &   -30220  & -2068  &   5737\\
1975    &   -30100  & -2013  &   5675\\
1980   &    -29992  & -1956  &   5604\\
1985      & -29873  & -1905   &  5500\\
1990  &     -29775 & -1848 &    5406\\
    1995& -29692& -1784&5306\\
    2000& -29619& -1728&5186\\
    2005& -29554& -1669&5077\\
    2010& -29496& -1586&4944\\
    2015& -29442& -1501&4797\\ 
    \bottomrule
    \end{tabular}
    \label{tab:my_label}
\end{table}
We denote geomagnetic (MAG) and geographic (GEO) coordinates as $Q_m = [x_m, y_m, z_m]^T$ and $Q_g = [x_g, y_g, z_g]^T$, respectively. In some coordinate systems the position $[x, y, z]^T$ is often defined by latitudes $\varphi$, longitude $\Theta$ and radial distance $R$, such that:
\begin{subequations}
\small
\begin{align}
    & x = R\cos(\varphi)\cos(\Theta), \quad 
    y = R\cos(\varphi)\sin(\Theta),  \quad 
    z = R\sin(\varphi) 
\end{align}
\end{subequations}%
\noindent And
\begin{subequations}
\small
\begin{align}
    & R = \sqrt{(x^2+y^2+z^2)}, \quad 
    \varphi = \arctan\left(\frac{z}{\sqrt{x^2+y^2}}\right), \quad
    \Theta=\left\{
            \begin{array}{ll}
             \arccos\left(\frac{x}{\sqrt{x^2+y^2}}\right), \ \text{if $y \geq 0$} \\
            -\arccos\left(\frac{x}{\sqrt{x^2+y^2}}\right), \ \text{otherwise}   \\
                \end{array}
              \right.
\end{align}
\end{subequations}%
\noindent Using this notation, as described in \cite{hapgood1992space}, the MAG coordinates are converted to the GEO coordinates using the following equation:
\begin{subequations}
\small
\begin{align}
   & Q_m = T \cdot Q_g
\end{align}
\end{subequations}%
\noindent where
\begin{subequations}
\small
\label{eq:GEO}
\begin{align}
& T = \langle \varphi-90\degree, Y\rangle \cdot \langle \Theta, Z\rangle \\
& \langle \varphi-90\degree, Y\rangle = \begin{bmatrix}
    \cos(\varphi-90\degree)& 0 & \sin(\varphi-90\degree)  \\
    0 & 1  & 0 \\
    -\sin(\varphi-90\degree) & 0 & \cos(\varphi-90\degree)
\end{bmatrix}, \ \ \langle \Theta, Z\rangle = \begin{bmatrix}
    \cos(\Theta)& \sin(\Theta) & 0 & \\
    -\sin(\Theta) & \cos(\Theta) & 0 \\
    0 & 0 & 1
\end{bmatrix}  \\
& \Theta = \arctan\left(\frac{h_1^1}{g_1^1}\right), \quad
\varphi = 90\degree-\arcsin\left(\frac{g_1^1\cos(\Theta) + h_1^1\sin(\Theta)}{g_1^0}\right)
\end{align}
\end{subequations}

\end{appendices}

\bibliographystyle{IEEEtran}
\bibliography{references.bib}

\end{document}